\documentclass[11pt]{article}
\usepackage{graphicx} % Required for inserting images
\usepackage{amsmath}
\usepackage{amsthm, amssymb}
\usepackage{amsfonts}
\usepackage[dvipsnames]{xcolor}
\usepackage{tikz,tikz-network}
\usepackage{tcolorbox}
\usepackage{graphicx}
\usepackage{physics}
\usepackage{dsfont}
\usepackage{xcolor}
\usepackage{appendix}
\usepackage{stmaryrd}
\usepackage{hyperref}
\usepackage{float}

\usepackage{authblk}
\usepackage{stmaryrd}
\usepackage[a4paper]{geometry}
\newtheorem{theorem}{Theorem}[section]
\newtheorem{proposition}[theorem]{Proposition}
\newtheorem{corollary}[theorem]{Corollary}
\newtheorem{lemma}[theorem]{Lemma}
\newtheorem{definition}[theorem]{Definition}
\newtheorem{remark}[theorem]{Remark}

\def\uni{\mathrm{uni}}
\def\ab{\mathrm{ab}}

\title{Eigenvalues of Maximal Abelian Covers}
\author{Wenbo Li, Michael Magee, Mostafa Sabri, Joe Thomas}

\date{\today}

\begin{document}

\maketitle

\begin{abstract}
We fully characterize the eigenvalues (flat bands) of the maximal abelian cover of a finite multi-graph in terms of the combinatorics of the base graph. This solves a problem of Higuchi and Nomura (2009, Problem 6.11). We use our new criterion to prove that the maximal abelian cover of any regular multi-graph has no eigenvalues, thereby proving a conjecture of (ibid., Conjecture 6.12). \\
In an appendix, we relate our criterion for eigenvalues of the maximal abelian cover to an existing criterion for eigenvalues of the universal cover.
\end{abstract}

\tableofcontents

\section{Introduction}\label{Sec:Intro}

 Let $G=(V,E)$ be a finite multi-graph. Among connected multi-graphs that are Galois covers of $G$ with abelian deck group, there is a maximal one that we denote by $G^\mathrm{ab}$ and refer to as the  \textbf{maximal abelian cover} of $G$. A Schr\"{o}dinger operator on $G$ is an operator $\mathcal{H}:\ell^2(V)\to \ell^2(V)$ whose matrix entries are 
 \begin{equation}\label{eq:schrodinger_op}
 \mathcal{H}_{u v} = \sum_{\substack{\text{directed edges $e$} \\ \text{from $u$ to $v$}}} w_e + \delta_{u v} \mathcal{V}_v ,
 \end{equation}
 where $w_e\in\mathbb{C}$ satisfy $w_e=\overline{w_{e^{-1}}}\neq 0$ and $V_v \in \mathbb{R}
 $, so $\mathcal{H}$ is self-adjoint. When $w_e\equiv 1$ and $\mathcal{V}_v\equiv 0$ for all $e\in \vec{E}$ and $v \in V$, $\mathcal{H}$ is the \textbf{adjacency operator} on $G$.

In this article we are interested in the Schr\"{o}dinger operator $\mathcal{H}^{\ab}$ on $\ell^2(G^\mathrm{ab})$ defined through the pull back of edge and vertex weights from $G$. This is a  Schr\"odinger operator on a $\mathbb{Z}^\nu$-periodic graph with periodic potential. It is known that the spectrum of $\mathbb{Z}^\nu$-periodic graphs in general consists of intervals of absolutely continuous spectrum together with potentially some eigenvalues of infinite multiplicity. These eigenvalues are known as ``flat bands'', in reference to the Bloch-Floquet theory which is used to analyze them \cite{Hi.No2009,Sa.Yo2023}. 

Understanding these eigenvalues is important for physics applications \cite{KFSH20}. It is easy to generate families of $\mathbb{Z}^\nu$-periodic graphs which have eigenvalues using various procedures, such as graph decorations. Guaranteeing however that a given periodic graph has no eigenvalues is a lot more challenging, even for abelian covers of very small graphs \cite{Sa.Yo2023}.

Every \(\mathbb{Z}^\nu\)-periodic graph is an Abelian cover of its quotient under the action of \(\mathbb{Z}^\nu\).
Maximal abelian covers are of great importance since their eigenvalues are common to all periodic graphs that are abelian covers of the same multi-graph (see the first paragraph of Section \ref{sec:prio}). Moreover, many well-studied periodic graphs arise as maximal abelian covers such as the integer, hexagonal and Lieb lattices. Among these examples, only the Lieb lattice has an eigenvalue. See \cite{Hi.No2009} and \cite[\S 8.3.V-VI]{Su2012} for more examples and background.

%For example, if $G$ is a bouquet then $G^{\mathrm{ab}}$ is an integer lattice, if $G$ is the graph on two vertices with three parallel edges, then $G^{\mathrm{ab}}$ is the hexagonal lattice. If $G$ is the figure eight graph on three vertices, then $G^{\mathrm{ab}}$ is the Lieb lattice. Among these examples, only the Lieb lattice has an eigenvalue. See \cite{Hi.No2009} and \cite[\S 8.3.V-VI]{Su2012} for more examples and background.

The first theorem of the paper is a complete geometric characterization of the eigenvalues of maximal abelian covers. This answers a more general version of a question of Higuchi and Nomura \cite[Problem 6.11]{Hi.No2009}. The (potential) eigenvalues of $\mathcal{H}^{\ab}$ are characterized in terms of 
roots of generalized matching polynomials of subgraphs of the graph $G$ (see Definition \ref{Def:generalisedmatchingpolynomial}). In the important special case that $\mathcal{H}$ is the adjacency operator, these polynomials reduce to the standard matching polynomials.

\begin{theorem}\label{thm:criteria}
    For any $\lambda\in \mathbb{R}$, $\lambda$ is an eigenvalue of $\mathcal{H}^{\ab}$ if and only if for any degree-2 subgraph $\gamma$ of $G$, $\lambda$ is a root of the generalized matching polynomial $m_{G\setminus\gamma}^\mathcal{H}$.
\end{theorem}

%Here we use the convention that the empty subgraph is $2$-regular, so $\lambda$ should also be a root of $m_G^{\mathcal{H}}$. 

%By a M\"{o}bius inversion argument, Theorem~\ref{thm:criteria} implies that  $\lambda$ is an eigenvalue of $\mathcal{H}^{\ab}$ if and only if it is an eigenvalue of $\mathcal{H}|_{G\setminus \gamma}$ for any degree-$2$ subgraph $\gamma$, see Remark~\ref{remark:cha=mat}. 

%Furthermore, Theorem~\ref{thm:criteria} gives a restriction on where the eigenvalues of $\mathcal{H}^{\ab}$ can lie, through a classical bound of Heilmann and Lieb \cite[Thm. 4.3]{He.Li1972} on the roots of the matching polynomial. More precisely, one may easily generalize the arguments of \cite{marcus2015interlacing,godsil1993algebraic} to extend the Heilmann-Lieb result to Schr\"odinger operators on multigraphs, and deduce that any eigenvalue of $\mathcal{H}^{\ab}$ must lie in the interval
%\[
%\left[-\rho,\rho\right],
%\]
%where $\rho$ is the spectral radius of $\mathcal{H}$ lifted to the universal cover--- see Corollary~\ref{cor:ram-bound}. For instance, when $\mathcal{H}$ is the adjacency operator of a $d$-regular (resp. $(c,d)$-biregular) graph, then $\rho = 2\sqrt{d-1}$ (resp. $\rho=\sqrt{c-1}+\sqrt{d-1}$).
Here we use the convention that the empty subgraph is $2$-regular, so $\lambda$ should also be a root of $m_G^{\mathcal{H}}$. By a M\"{o}bius inversion argument, Theorem~\ref{thm:criteria} implies that  $\lambda$ is an eigenvalue of $\mathcal{H}^{\ab}$ if and only if it is an eigenvalue of $\mathcal{H}|_{G\setminus \gamma}$ for any degree-$2$ subgraph $\gamma$, see Remark~\ref{remark:cha=mat}. 
Furthermore, in the case that $\mathcal{H}$ is the adjacency operator on $G$, and $G$ is $d$-regular, a result of Heilmann and Lieb \cite[Thm. 4.3]{He.Li1972} says that the roots of $m_{G}^\mathcal{H}$ are in 
\[
[ -2\sqrt{d-1} , 2 \sqrt{d-1} ]
\]
--- which is the $\ell^2$ spectrum of the adjacency operator of the universal cover of $G$. By an easy modification to the arguments of \cite{marcus2015interlacing,godsil1993algebraic} this extends to Schr\"{o}dinger operators on a general multi-graph where $2\sqrt{d-1}$ is replaced by the spectral radius of $\mathcal{H}$ lifted to the universal cover --- see Corollary \ref{cor:ram-bound}. In particular, Theorem \ref{thm:criteria} implies that all eigenvalues of $\mathcal{H}^{\mathrm{ab}}$ lie in this interval too.

%Furthermore, in the case that $\mathcal{H}$ is the adjacency operator on $G$, and $G$ is $d$-regular, a result of Heilmann and Lieb \cite[Thm. 4.3]{He.Li1972} says that the roots of $m_{G}^\mathcal{H}$ are in 
%\[
%[ -2\sqrt{d-1} , 2 \sqrt{d-1} ]
%\]
%--- which is the $\ell^2$ spectrum of the adjacency operator of the universal cover of $G$ --- and hence Theorem  \ref{thm:criteria} implies that all eigenvalues of $\mathcal{H}^{\ab}$ are in this interval too. For the extension to Schr\"{o}dinger operators on a general graph, see Corollary \ref{cor:ram-bound}.

A direct consequence of Theorem \ref{thm:criteria} is that the maximal Abelian cover of \(G\) has no eigenvalues whenever \(G\) admits a \(2\)-factor, i.e. a degree-$2$ subgraph containing all vertices of \(G\). This holds for all even regular graphs by Petersen's theorem, and it also applies to odd regular graphs with few bridges, see Lemma \ref{Lemma: keylemma} for a more precise statement.
Nonetheless, we are able to use the criterion of Theorem \ref{thm:criteria} to still prove that $\mathcal{H}^{\ab}$ has no eigenvalues for all finite regular multi-graphs.
\begin{theorem}\label{thm:regular}
 If $G$ is regular, then $\mathcal{H}^{\ab}$ has no eigenvalues.
\end{theorem}

Theorem~\ref{thm:regular} is the second result of this paper, and is our main contribution. It confirms one of the main conjectures about the spectral structure of maximal abelian covers of regular graphs \cite[Conj. 6.12]{Hi.No2009}, the remaining one being the lack of spectral gaps \cite{HS,HS2}. Our answer is in fact more general as it allows for Schr\"odinger operators. 

The proof of Theorem \ref{thm:regular} involves, after a little Floquet theory, a delicate combinatorial argument that carefully studies the structure of the bridges within $G$ and the blocks that they decompose the graph into, together with a mechanism to find degree-$2$ subgraphs on bridge-less components of the graph containing all vertices not connected to bridges. This novel combinatorial argument combines well with the well-known recursion relation satisfied by the generalized matching polynomials when deleting vertices from a graph (Lemma \ref{lem:recursion}). Because the argument relies only on these two things we believe it to be robust and more generally useful and interesting.  We refer to Section \ref{sec:proof_overview} for an in-depth overview of the proof of Theorem \ref{thm:regular}.

Although Theorem~\ref{thm:regular} concerns regular graphs, it can be combined with Theorem~\ref{thm:criteria} to deduce some simple consequences for non-regular graphs. In particular, we can prove that $\mathcal{H}^{\ab}$ has no eigenvalues if e.g. $G$ is obtained by attaching a leaf to each vertex of an even-regular graph, or adding self-loops to some vertices of a regular graph. Furthermore, by using the ideas of the proof of Theorem \ref{thm:regular} we were able to prove $\mathcal{H}^{\ab}$ has no eigenvalues other than zero for many bi-regular graphs (for vertex potentials equal to zero); the bi-regular case in general remains open and is a natural follow on question from our paper.

%Let us mention that due to an inheritance relation of eigenvalues \cite{Hi.No2009}, see \S\,\ref{sec:prio} below, Theorem~\ref{thm:criteria} is useful to all $\mathbb{Z}^d$-periodic graphs $\Gamma$. It gives a new sufficient condition for $\Gamma$ to have a flat band, in terms of its quotient graph from the $\mathbb{Z}^d$ action. The converse is not true: many regular periodic graphs have flat bands (see e.g. \cite[Fig. 1]{Sa.Yo2023}), and Theorem~\ref{thm:regular} implies that they come from a different mechanism.

We conclude the introduction with a brief discussion of the eigenvalues of the pull-back of $\mathcal{H}$ to the \textbf{universal cover} $\tilde{G}$ of $G$, which we denote by $\mathcal{H}^{\uni}$. Recent work of Banks, Garza-Vargas and Mukherjee \cite{banks2022point} characterizes the eigenvalues of $\mathcal{H}^{\uni}$ in terms of the combinatorics and spectra of \textbf{Aomoto subsets of $G$},  building upon earlier work of Aomoto \cite{Ao1991}, see also \cite{Salez20}. This criterion is stated as the B-GV-M criterion in the appendix to this paper.

Both Banks, Garza-Vargas, Mukherjee \cite[Theorem 3.2]{banks2022point} and Arizmendi, Cébron, Speicher, and Yin \cite[Proposition 7.6]{arizmendi2024universality} prove that any eigenvalue of the universal cover of a finite graph \(G\) is also an eigenvalue of \(G\) itself.   Although not explicitly stated, the argument from \cite{arizmendi2024universality} also shows that any eigenvalue of $\mathcal{H}^{\uni}$ is an eigenvalue of $\mathcal{H}^{\mathrm{ab}}$. This suggests a direct relation between the B-GV-M criterion and the one in Theorem \ref{thm:criteria}. In Appendix~\ref{ap} we prove that indeed, if the graph $G$ satisfies the B-GV-M criterion (so $\mathcal{H}^\uni$ has an eigenvalue), then it also satisfies our criterion (so $\mathcal{H}^\ab$ has the same eigenvalue).

As a final remark, despite trying, we could not find a counterexample $G$ to the possibility that when $\mathcal{H}$ is the adjacency operator of $G$, every eigenvalue of $\mathcal{H}^{\mathrm{ab}}$ is an eigenvalue of $\mathcal{H}^{\uni}$ --- and neither could we prove it. It would be highly desirable to resolve this matter one way or the other. 

\subsection{Prior work}\label{sec:prio}

When $G$ has a $2$-factor, Theorem~\ref{thm:criteria} immediately implies that $\mathcal{H}^{\ab}$ has no eigenvalues. This particular consequence was proved by Higuchi and Nomura before \cite{Hi.No2009} in the case of  Laplacians by a different method. Namely, they showed that eigenvalues of the maximal abelian cover descend to all abelian covers of the same base graph. They then used the $2$-factor to construct a particular $\mathbb{Z}$-periodic cover which has no eigenvalues, thereby proving the result.

Maximal abelian covers are bipartite, hence the spectrum of their adjacency operator symmetric. It was shown in \cite{HS,HS2} that if all degrees of the graph are even, then the normalized Laplacian on the maximal abelian cover has no spectral gaps: the spectrum fills the whole interval $[0,2]$. The authors also prove this for certain families of odd-regular graphs. An eigenvalue of $\mathcal{H}^{\ab}$ can never be at the top of the spectrum, so at least for these examples, the point spectrum is completely embedded in the continuous spectrum. 

As we previously mentioned, only limited results are known about flat bands of periodic graphs, despite the very large physics literature on the topic. We refer to \cite{Sa.Yo2023} for an overview. An interesting recent result \cite{FL} confirms the physical intuition that flat bands are rare, by showing that for a given infinite periodic graph, the set of weights such that the corresponding Schr\"odinger operator has no eigenvalues is generic.

\subsection{Overview of the paper}

The remainder of the paper is organized as follows. In Section~\ref{Sec:constructionofmaxabelian}, for a given finite graph $G$, we construct an explicit periodic graph $G^{\mathrm{per}}$ whose connected components are all isomorphic to the maximal abelian cover $G^{\mathrm{ab}}$. Despite being a very high-dimensional object in general, $G^{\mathrm{per}}$ has the advantage of possessing a Floquet matrix which is easier to understand using generalized matching polynomials, whose key properties we review in Section~\ref{Sec:thecriteria} before proving our criterion, Theorem~\ref{thm:criteria}. 
Section~\ref{Sec:atomthm} is dedicated to the proof of Theorem~\ref{thm:regular}. We begin with an overview of the proof in \S\,\ref{sec:proof_overview}. In \S\,\ref{Section:TypeII}, we construct degree-$2$ subgraphs containing all vertices of maximal degree in bridge-less multi-graphs (Proposition~\ref{prop: smoothprocedure}). This result is crucial for our proof and may be of broader interest. In \S\,\ref{section:Type I} we study the bridge-blocks of maximum degree strictly less than $d$ via Lemmas~\ref{lemma:randomleafavoideigenvalue} and~\ref{lemma:lemmaforcontradiction}, which show how removing leaves prohibits $\lambda$ from being a root of the generalized matching polynomial. The relation between leaves and degree-$2$ subgraphs uses Proposition~\ref{prop: smoothprocedure} and is explained later in \S\,\ref{sec:prooreg}, where all ingredients are finally combined in an inductive procedure (Lemma~\ref{lemma: mainlemma}) to prove Theorem~\ref{thm:regular}. The paper ends with Appendix~\ref{ap}, where we discuss the B-GV-M criterion for universal covers and its relation to $\mathcal{H}^{\mathrm{ab}}$.

\subsection{Notation}
We write $\mathbb{N}=\{1, 2 ,\cdots\}$ and for any $n\in \mathbb{N}$, we write $[ 1, n]=\{1,2,\cdots,n\}$. The cardinality of a set $I$ is written as $|I|$. For any sets $A$ and $B$, the set $A\setminus B$ is defined as 
\[
A\setminus B=\{x:x\in A, x\notin B\}.
\]
Note that $B$ is not necessarily a subset of $A$.

In this paper we work with  multi‑graphs allowing both \textbf{self‑loops} and \textbf{multi-edges}. For a graph $G=(V,E)$, $V$ is the set of vertices and $E$ is the set of undirected edges. We write $\vec{E}$ for the directed edge set of $G$, that is, it contains all edges of $E$ with both orientations. We also sometimes write \(V_G\), \(E_G\) and \(\vec{E}_G\) if the graph $G$ under consideration is ambiguous. For each edge $e\in\vec{E}$, 
its \textbf{origin} and \textbf{terminus} are denoted as $o(e)$ and $t(e)$ respectively, and 
we write $e^{-1}$ to be the same combinatorial edge but with the opposite direction. The maximal abelian cover of a graph $G$ is denoted by $G^{\mathrm{ab}}$ and the universal cover by $\tilde{G}$.

For a graph $G=(V,E)$ and a set $S\subset V$, the \textbf{induced subgraph} of $G$ on $S$ is denoted by $G[S]$ while the induced subgraph on $V\setminus S$ is denoted by $G\setminus S$. Similarly, if $H$ is a subgraph of $G$, we write $G\setminus H$ to denote the induced subgraph on $V\setminus V_H$. A \textbf{$k$-factor} of a graph $G$ is a $k$-regular spanning subgraph of $G$, i.e., a $k$-regular subgraph that shares the same vertex set as $G$. 

\subsection*{Acknowledgments}
We thank Charles Bordenave for informing us of the paper \cite{arizmendi2024universality}. W.L. has received funding from the National Natural Science Foundation of China (Grant No.\,123B2013). M.M. has received funding from the European Research
Council (ERC) under the European Union’s Horizon 2020 research and innovation programme (grant agreement No 949143). This work was funded in part by a Philip Leverhulme prize (M.M.) awarded by the Leverhulme Trust. J.T. has received funding from the Leverhulme Trust through a Leverhulme Early Career Fellowship (Grant No. ECF-2024-440).

\section{Periodic graphs and the maximal abelian cover}\label{Sec:constructionofmaxabelian}

Let $G=(V,E)$ be a finite multi-graph. In this section, we construct an explicit periodic graph $G^{\mathrm{per}}$ that is a union of at most a countable number of copies of the maximal abelian cover $G^{\mathrm{ab}}$ of $G$. We then show by means of Floquet theory that the $\ell^2$ eigenvalues of $\mathcal{H}^{\ab}$ coincide with the eigenvalues of an associated $n\times n$ matrix $\mathcal{H}(z)$ which are constant in $z\in \mathbb{T}^{|E|}$.

\subsection{Construction of \texorpdfstring{\(G^{\mathrm{per}}\)}{}}

Given $G=(V,E)$, let $\vec{E}$ be the collection of directed edges of $E$, so that each edge appears in $\vec{E}$ with both orientations and let $\vec{E}_0\subseteq \vec{E}$ be a choice of orientation on each edge in $E$. That is, every edge in $E$ appears with exactly one of its directions in $\vec{E}_0$. Let $F(E)$ and $F_{ab}(E)$ be the free group and free abelian group generated by $\vec{E}_0$ (different choices of $\vec{E}_0$ give isomorphic groups so we drop the dependence in the notation). The inverse of a generator in either of these groups is the same edge but with the opposite orientation, hence the set of generators and their inverses is identified with $\vec{E}$. The group product of $e_1,e_2 \in F(E)$ (resp. $F_{ab}(E)$) is denoted by $e_1e_2$ (resp. $e_1+e_2$). The canonical abelianization map from $F(E)$ to $F_{ab}(E)$ is denoted as $\pi_0$. That is, $\pi_0$ maps each generator of $F(E)$ to the corresponding generator in $F_{ab}(E)$ and extends as a homomorphism (e.g. $\pi_0(e_1e_2e_1^{-1}e_3)=e_2+e_3$).

 The periodic graph $G^{\mathrm{per}}$ is then constructed as follows.

\begin{definition}
     Let $G^{\mathrm{per}}$ be the graph with vertex set $V_{G^{\mathrm{per}}}= F_{ab}(E)\times V$.  For $(u,i), (v,j) \in  F_{ab}(E)\times V$ where $u,v\in F_{ab}(E), i,j \in V$, there is an edge between $(u,i)$ and $(v,j)$ if and only if $v=u+e$ and $o(e)=i, t(e)=j$ for some $e\in \Vec{E}$. 
\end{definition}

In the same spirit, we define the following graph $G^{\mathrm{uni}}$.

\begin{definition}
    Let $G^{\mathrm{uni}}$ be the graph with vertex set $V_{G^{\mathrm{uni}}}=F(E)\times V$, and for $(u,i), (v,j) \in F(E)\times V$ where $u,v\in F(E), i,j \in V$, there is an edge between $(u,i)$ and $(v,j)$ if and only if $v=ue$ and $o(e)=i, t(e)=j$ for some $e\in \Vec{E}$. 
\end{definition}

Note that the abelianization $\pi_0:F(E)\to F_{ab}(E)$ induces a covering map $\pi:G^{\mathrm{uni}}\to G^{\mathrm{per}}$ given by $\pi((u,i))=(\pi_0(u),i)$. In the case where \(G\) is a \textbf{bouquet} of degree \(2d\) (a graph with a single vertex and \(d\) self-loops), \(G^{\mathrm{per}}\) and \(G^{\mathrm{uni}}\) are, respectively, the \(d\)-dimensional integer lattice and the \(2d\)-regular tree, with \(\pi\) being the universal covering map onto the lattice. In particular, they coincide with $G^\mathrm{ab}$ and $\tilde{G}$ respectively. A similar result holds for general \(G\), as shown in the next section.

\subsection{Realizing \texorpdfstring{$G^{\mathrm{ab}}$}{} inside \texorpdfstring{$G^{\mathrm{per}}$}{}}

Let \(G\) be a finite multi-graph with fundamental group \(\pi_1(G)\).  
The maximal abelian cover \(G^{\mathrm{ab}}\) is the cover of \(G\) corresponding to the subgroup \([\pi_1(G), \pi_1(G)]\) under the Galois correspondence.  
Equivalently,
\[
G^{\mathrm{ab}} \cong \tilde{G} / [\pi_1(G), \pi_1(G)],
\]
the quotient of \(\tilde{G}\) by the action of \([\pi_1(G), \pi_1(G)]\). We first prove that \(G^\mathrm{per}\) consists of at most countably many copies of \(G^\mathrm{ab}\).

\begin{proposition}
\label{prop:components}Every connected component of $G^{\mathrm{per}}$
is isomorphic to $G^{\mathrm{ab}}$.
\end{proposition}

\begin{proof}
Fix $(u,i)\in G^{\mathrm{uni}}$ arbitrarily and let $C$ be its connected component. Note that the connected components of $G^\mathrm{uni}$ are mapped surjectively onto the connected components of $G^\mathrm{per}$ under the covering map $\pi$. Let $\tilde{G}$ be the standard construction of the universal cover as non-backtracking walks based at the vertex $i$. That is, $\tilde{G}$ has vertex set
\[
\{\{e_1,\dots, e_n\}:\{e_1,\dots, e_n\} \text{ is a non-backtracking walk starting from}\ i\},
\]
and there is an edge between two such non-backtracking walks $P_1$ and $P_2$ if one is a one-step extension of the other. We construct a map $\iota$ from $\tilde{G}$ to $G^{\mathrm{uni}}$ by 
\[
\iota(\{e_1,e_2,\dots, e_n\})=(ue_1e_2\cdots e_n, t(e_n)),
\]
where on the right-hand side we see the edges $e_1,\ldots,e_n$ as their corresponding elements in the free group $F(E)$. It is straightforward to see that $\iota$ gives rise to a graph isomorphism between $\tilde{G}$ and $C$.  

Now let
    \[
		\pi_1(G)=\Big\{ \gamma=\prod_{k=1}^ne_k: \{e_1,e_2,\dots, e_n\}\ \text{is a non-backtracking walk from}\ i\ \text{to}\ i\ \text{in}\ G\Big\}\leq F(E).
		\]
This group is isomorphic to the fundamental group of the graph (based at the vertex $i$) because each homotopy class of closed walk from $i$ has a unique non-backtracking representative. 

The fundamental group $\pi_1(G)$ has a standard action on the universal cover $\tilde{G}$ given by 
    \[\gamma \cdot \{e_1,\ldots,e_n\} := \text{the non-backtracking walk that the walk \ }\{\gamma,e_1,\ldots,e_n\}\ \text{reduces to}, \] 
and hence by intertwining with $\iota$ we get an action of $\pi_1(G)$ on $C$ defined by 
\[
\gamma \cdot (v,j):= (u \gamma u^{-1}v, j)
\]
for any $\gamma \in \pi_1(G)$ and $(v,j)\in C$. This is because for any non-backtracking walk $\{e_1,\ldots,e_n\}$ starting at $i$, and any $\gamma\in\pi_1(G)$ we have
    \begin{align*}
        \iota(\gamma\cdot\{e_1,\ldots,e_n\})&=(u\gamma e_1\cdots e_n,t(e_n))\\
        &=(u\gamma u^{-1}ue_1\cdots e_n,t(e_n))\\
        &=\gamma\cdot\iota(\{e_1,\ldots,e_n\}).
    \end{align*} 

Recall that $G^\mathrm{ab}\cong\tilde{G}/[\pi_1(G),\pi_1(G)]\cong C/[\pi_1(G),\pi_1(G)]$. We claim that any two vertices in $C$ are in the same $[\pi_1(G),\pi_1(G)]$-orbit if any only if their image under the covering map $\pi$ is the same. This means that the quotient of $C$ by the action of $[\pi_1(G),\pi_1(G)]$ is graph isomorphic to the corresponding connected component in $G^\mathrm{per}$, hence proving the proposition. 
To prove this claim, let $(u_1,i_1)$ and $(u_2,i_2)$ be any two vertices in $C$ so that by definition, 
$$u_1=u \prod_{k=1}
^{n_1}e_{k;1}\,\,\,\text{and}\,\,\,u_2=u \prod_{k=1}
^{n_2}e_{k;2},$$
where $\{e_{1;1},e_{2;1},\cdots,e_{n_1;1}\}$ and  $\{e_{1;2},e_{2;2},\cdots,e_{n_2;2}\}$ are two non-backtracking walks in $G$ from $i$ to $i_1$ and $i_2$ respectively. If $\pi((u_1,i_1))=\pi((u_2,i_2))$ then we have $i_1=i_2$ and $\pi_0(u_1u_2^{-1})=0$, which means first of all that
$$L_0:=\{ e_{1;1},e_{2;1},\cdots,e_{n_1;1},e_{n_2;2}^{-1}, e_{n_2-1;2}^{-1}, \cdots e_{1;2}^{-1}\}$$
is a closed walk from $i$ to $i$ in $G$ and hence 
\[
\gamma_0:=\prod_{k=1}
^{n_1}e_{k;1}\,\prod_{k=0}
^{n_2-1}e^{-1}_{n_2-k;2}\in\pi_1(G).
\]
And secondly, 
    \[
		\sum_{k=1}^{n_1}e_{k;1}+\sum_{k=1}^{n_2}e^{-1}_{k;2}=0
		\]
in $F_{ab}(E)$. The restriction of $\pi_0$ to $\pi_1(G)$ is precisely the Hurewicz homomorphism (see for example \cite[Section 5.3]{Su2012}) which coincides with the abelianization of $\pi_1(G)$. Thus, the kernel of $\pi_0|_{\pi_1(G)}$ is exactly $[\pi_1(G),\pi_1(G)]$. So, since $\pi_0(\gamma_0)=0$, we must have $\gamma_0\in [\pi_1(G),\pi_1(G)]$ and hence since $\gamma_0$ sends $(u_2,i_2)$ to $(u_1,i_1)$ we prove one part of the claim.

For the converse, if $\gamma\cdot (u_1,j_1)=(u_2,j_2)$ and $\gamma \in [\pi_1(G),\pi_1(G)]$, then $j_1=j_2$ and $u\gamma u^{-1}u_1=u_2$. Thus,
    $$\pi_0(u_2)=\pi_0(u\gamma u^{-1})+\pi_0(u_1)=\pi_0(u)+\pi_0(u^{-1})+\pi_0(\gamma)+\pi_0(u_1)=\pi_0(u_1).$$
It follows that $\pi_0(u_2)=\pi_0(u_1)$ and $\pi(u_1,j_1)=\pi(u_2,j_2)$ as required.
\end{proof}

\subsection{Eigenvalues of \texorpdfstring{$G^{\mathrm{ab}}$}{} and \texorpdfstring{$G^{\mathrm{per}}$}{}}

We now consider a Schr\"{o}dinger operator $\mathcal{H}$ on $G$ as in \eqref{eq:schrodinger_op}.
Denote the pull-back of $\mathcal{H}$ to $G^{\mathrm{ab}}$ and $G^{\mathrm{per}}$ as $\mathcal{H}^{\mathrm{ab}}$ and $\mathcal{H}^{\mathrm{per}}$, respectively. For each $z\in\mathbb{T}^{|\vec{E}_0|}$ define the matrix $\mathcal{H}(z)$ by

$$\mathcal{H}(z)=\mathcal{A}(z)+\mathcal{V},$$
where 
$$\mathcal{A}(z)=\sum_{e\in \vec{E}_0}z_ew_eA_e+\overline{z_e}\overline{w_e}A_{e^{-1}}.$$
The following proposition relates the spectrum of $\mathcal{H}(z)$ to $\mathcal{H}^{\ab}$.
\begin{proposition}
\label{prop:eigenvalue-correspondence}The maximal abelian cover $G^{\mathrm{ab}}$
of $G$ has an eigenvalue $\lambda$ if and only if $\lambda$ is
an eigenvalue of $\mathcal{H}(z)$ for all $z=(z_{e})_{e\in \vec{E}_0}\in\mathbb{T}^{|\vec{E}_0|}$.
\end{proposition}

\begin{proof}
We will use \cite[Lemmas 2.1--2.3]{Sa.Yo2023} by reframing $G^\mathrm{per}$ explictly as a $\mathbb{Z}^{|E|}$-periodic graph in the same language as [ibid.]. For this construction, label the vertices of $G$ by $1,\ldots,n$ and define the fundamental cell to be 
    \begin{align*}
        V_f=\{(0,i)\in\mathbb{R}^{|E|}\times\mathbb{R}:i=1,\ldots,n\}.
    \end{align*}
The vertex set of the periodic graph is then $\mathbb{Z}^{|E|}\times V_f$ and the periodic coordinate vectors are simply the unit vectors $(f_e)_{e\in E}$ in the first $|E|$ coordinates. Each coordinate direction in the first $|E|$ coordinates corresponds to a generator of $F_\mathrm{ab}(E)$: the coordinate axis labelled by an edge in $E$ corresponds to its unique oriented edge in $\vec{E}_0$. 

The edges of the periodic graph are defined by joining $(a,i)\in \mathbb{Z}^{|E|}\times V_f$ to $(b,j)\in \mathbb{Z}^{|E|}\times V_f$ if either 
    \begin{enumerate}
        \item there exists $e\in\vec{E}_0$ such that $e$ joins $i$ to $j$ and $b=a+f_e$,
        \item or there exists $e\in\vec{E}_0$ such that $e$ joins $j$ to $i$ and $b=a-f_e$.
    \end{enumerate}
By construction, this graph is $\mathbb{Z}^{|E|}$ periodic and isomorphic to $G^\mathrm{per}$. The operator $\mathcal{H}$ lifts to $\mathcal{H}^\mathrm{per}$ on $G^\mathrm{per}$ and hence to this periodic graph. Moreover,  the Floquet matrix of this operator is precisely \(\mathcal{H}(z)\). So by \cite[Lemmas 2.1--2.3]{Sa.Yo2023}, $\lambda$ is an eigenvalue of $\mathcal{H}^\mathrm{per}$ if and only if it is an eigenvalue of $\mathcal{H}(z)$ for all $z=(z_{e})_{e\in \vec{E}_0}\in\mathbb{T}^{|\vec{E}_0|}$.

If $\lambda$ is an eigenvalue of $\mathcal{H}^{\mathrm{ab}}$ then it is clearly
an eigenvalue of $\mathcal{H}^{\mathrm{per}}$, since any corresponding $\ell^{2}$ eigenfunction can be realized as an $\ell^{2}$ function on a connected component of $G^{\mathrm{per}}$ by Proposition \ref{prop:components}, upon which $\mathcal{H}^{\ab}$ has the same action as $\mathcal{H}^{\mathrm{per}}$, and we simply let the eigenfunction be $0$ on any other connected components of $G^\mathrm{per}$. Conversely, if $\lambda$ is an eigenvalue of $\mathcal{H}^{\mathrm{per}}$ then the restriction of any corresponding eigenfunction to any connected component of $G^\mathrm{per}$ is also an eigenfunction with the same eigenvalue. Choose any connected component upon which the restriction is non-zero, then since by Proposition \ref{prop:components} the connected component is isomorphic to $G^{\mathrm{ab}}$, we can realize the eigenfunction as an eigenfunction on $G^{\mathrm{ab}}$ with the same eigenvalue. 
\end{proof}

\section{Criterion for eigenvalues in the maximal abelian cover}\label{Sec:thecriteria}

In this section we prove Theorem~\ref{thm:criteria}, our criterion for the eigenvalues of the maximal abelian cover. 

\subsection{Matchings and the generalized matching polynomial}\label{sec:matrev}

We first review several results concerning generalized matching polynomials, as studied in e.g. \cite{Av.Ma2007,He.Li1972, Ku.Wong2013}. In these prior works, the results are stated for simple graphs only and so for the sake of completeness and to maintain notational consistency throughout the article, we will prove the required extensions to multi-graphs (despite causing no additional difficulties). 

\begin{definition}
Let $G=(V,E)$ be a finite multi-graph. A \textbf{matching} of $G$ is a degree-$1$ subgraph of $G$. A \textbf{$k$-matching} is a matching with exactly $k$ edges and we see the empty graph as the unique \(0\)-matching. A \textbf{perfect matching} is a matching whose vertex set is $V$. Denote the collection of all matchings and perfect matchings of $G$ as $\mathcal{M}(G)$ and $\mathcal{PM}(G)$, respectively.
\end{definition}

 Given a matching $M$, we let $V_M\subseteq V$ denote its vertex set and $E_M\subseteq E, \vec{E}_M\subseteq \vec{E}$ be its edge and directed edge sets respectively. Recall the definition of the Schr\"{o}dinger operators $\mathcal{H}$ as in equation \eqref{eq:schrodinger_op}.

\begin{definition}\label{Def:generalisedmatchingpolynomial}
The generalized matching polynomial of $G$ with respect to $\mathcal{H}$, denoted as $m_G^{\mathcal{H}}$, is defined as
\[
m_G^{\mathcal{H}}(x)=\sum_{M\in \mathcal{M}(G)}(-1)^{\frac{|V_M|}{2}}\prod_{e\in \vec{E}_M}w_e\cdot \prod_{u\in V\setminus V_M}\left(x-\mathcal{V}_u\right).
\]
If \(G_0\) is a subgraph of \(G\), we abuse notation by denoting the generalized matching polynomial of \(G_0\) with the restriction of \(\mathcal{H}\) as  
\(
m^\mathcal{H}_{G_0} \equiv m^{\mathcal{H}|_{G_0}}_{G_0}.
\)
\end{definition}

Note that in the case where \(|V|=n\) and $\mathcal{V}_u\equiv 0$, $w_e\equiv 1$, we have 
$$m_G^{\mathcal{H}}(x)=\sum_{M\in \mathcal{M}(G)}(-1)^{\frac{|V_M|}{2}} x^{|V\setminus V_M|}=\sum_{k=0}^{\lfloor\frac{n}{2}\rfloor} (-1)^k\,\#\{k\text{-matchings of}\ G\}\cdot x^{n-2k},$$
which is the standard matching polynomial of the graph $G$ corresponding to the adjacency operator. See \cite{godsil1993algebraic,Gut17} for more background on this polynomial.

\begin{remark}
    The incorporation of the edge weights in Definition \ref{Def:generalisedmatchingpolynomial} is not the same as in the weighted matching polynomials defined by Heilmann and Lieb \cite[Section 2]{He.Li1972}. Instead, our weighted matching polynomial corresponds to the Heilmann-Lieb definition for the weighted graph whose edge weights are $|w_e|^2$ on the (undirected) edge $e$ which is natural when computing determinants later.
\end{remark}

The generalized matching polynomial satisfies the following recursion relation. This was shown for simple graphs without vertex weights in \cite{He.Li1972} and for simple graphs in \cite{Av.Ma2007}.

\begin{lemma}[Recursion relation of matching polynomials]
\label{lem:recursion}
For any $v\in V$, we have
   \begin{equation*}
       m_G^{\mathcal{H}}(x)=\left(x-\mathcal{V}_v\right)\cdot m_{G\setminus \{v\}}^{\mathcal{H}}(x)-\sum_{u\sim v, u\neq v} \left(\sum_{e'\in \vec{E};\, o(e)=v,\,t(e)=u}|w_{e}|^2\right)\cdot m_{G\setminus \{v,u\}}^{\mathcal{H}}(x).
   \end{equation*}
\end{lemma}

\begin{proof}
    The proof proceeds by considering whether the vertex $v$ is in the vertex set of a given matching. Namely, we have
    \begin{align*}
        m_G^{\mathcal{H}}(x)
        %&=\sum_{M\in \mathcal{M}(G)}(-1)^{\frac{|V_M|}{2}}\prod_{e'\in \vec{E}_M}w_{e'}\cdot \prod_{i\in V\setminus V_M}\left(x-\mathcal{V}_i\right)\\
        &=\sum_{M\in \mathcal{M}(G),\, v\notin V_M}(-1)^{\frac{|V_M|}{2}}\prod_{e'\in \vec{E}_M}w_{e'}\cdot \prod_{i\in V\setminus V_M}\left(x-\mathcal{V}_i\right)\\
        &\quad +\sum_{M\in \mathcal{M}(G),\,v\in V_M}(-1)^{\frac{|V_M|}{2}}\prod_{e'\in \vec{E}_M}w_{e'}\cdot \prod_{i\in V\setminus V_M}\left(x-\mathcal{V}_i\right)=:S_1 + S_2.
        \end{align*}
        We simplify the first sum \(S_1\) as follows:
        \begin{align*}
        S_1&=\left(x-\mathcal{V}_v\right)\cdot\sum_{M\in \mathcal{M}(G\setminus\{v\})}(-1)^{\frac{|V_M|}{2}}\prod_{e'\in \vec{E}_M}w_{e'}\cdot \prod_{i\in \left(V\setminus\{v\}\right)\setminus V_M}\left(x-\mathcal{V}_i\right)\\
        &=\left(x-\mathcal{V}_v\right)\cdot m_{G\setminus \{v\}}^{\mathcal{H}}(x).
    \end{align*}
		If $v\in M$, it is matched to one of its neighbors through an edge $e$, which cannot be a self-loop. We thus have
    \begin{align*}
        S_2&=\sum_{\substack{e\in\vec{E},\, o(e)=v,\\ e\ \text{is not a self-loop}}}\ \sum_{\substack{M\in \mathcal{M}(G),\\ v\in V_M,\, e\in \vec{E}_M}}(-1)^{\frac{|V_M|}{2}}\prod_{e'\in \vec{E}_M}w_{e'}\cdot \prod_{i\in V \setminus V_M}\left(x-\mathcal{V}_i\right).
	\end{align*}
    After factoring out the contributions of $w_e$ and $w_{e^{-1}}$ in the interior sum, the remaining terms in the product of $w_{e'}$ run over the same matching but with $e=\{v,t(e)\}$ removed. Hence,
      \begin{align*} 
			S_2 =&\sum_{\substack{e\in\vec{E},\, o(e)=v,\\ e\ \text{is not a self-loop}}}\,|w_{e}|^2 \sum_{M\in \mathcal{M}(G\setminus\{v, t(e)\})}(-1)^{\frac{|V_M|}{2}+1}\prod_{e'\in \vec{E}_M}w_{e'}\cdot \prod_{i\in \left(V\setminus\{v, t(e)\}\right)\setminus V_M}\left(x-\mathcal{V}_i\right)\\
        =&\sum_{u\sim v, u\neq v}\, \sum_{e\in \vec{E};\, o(e)=v,\,t(e)=u}|w_{e}|^2 \sum_{M\in \mathcal{M}(G\setminus\{v, u\})}(-1)^{\frac{|V_M|}{2}+1}\prod_{e'\in \vec{E}_M}w_{e'}\cdot \prod_{i\in \left(V\setminus\{v, u\}\right) \setminus V_M}\left(x-\mathcal{V}_i\right)\\
        =&-\sum_{u\sim v, u\neq v}\ \left(\sum_{e\in \vec{E};\, o(e)=u,\,t(e)=v}|w_{e}|^2 \right) \cdot m_{G\setminus \{u,v\}}^{\mathcal{H}}(x).
    \end{align*}
		where the exponent on $-1$ changed because the new matching contains two fewer vertices.    
\end{proof}

\begin{remark}
     We will mostly use the recursion relation for the case when $v$ is degree $1$ in which case it simplifies to 
     $$ m_G^{\mathcal{H}}(x)=\left(x-\mathcal{V}_v\right)\cdot m_{G\setminus \{v\}}^{\mathcal{H}}(x)- |w_{e}|^2 \cdot m_{G\setminus \{u,v\}}^{\mathcal{H}}(x),$$
    where $u$ is the neighbour of $v$ and \(e\) is the directed edge joining $u$ and $v$. Later, we will write \(\left|w_{e}\right|\) simply as \(\left|w_{\{u,v\}}\right|\).

\end{remark}

Towards a computation of the determinant of $\mathcal{H}(z)$ as is required by Proposition \ref{prop:eigenvalue-correspondence}, we introduce the following definition of an oriented degree-$2$ subgraph.
\begin{definition}
    Let $G=(V,E)$ be a finite multi-graph. An \textbf{oriented degree-$2$ subgraph} $\gamma=\big(V_\gamma, \vec{E}_\gamma \big)$ of $G$ is a degree-$2$ subgraph of $G$ with directed edges $\vec{E}_\gamma\subseteq \vec{E}$ that give a consistent orientation to each of its connected components. Denote the collection of oriented degree-$2$ subgraphs of $G$ by $\Gamma=\Gamma_G$.
\end{definition}
 For an oriented degree-2 subgraph $\gamma=\left(V_\gamma, \vec{E}_\gamma \right)$, define 
     \(w_\gamma=\prod_{e\in \vec{E}_\gamma}w_{e}\).
The determinant of $\mathcal{A}$ is then given by the following.
\begin{lemma}\label{lemma:edgeweightexpansion}
     For $\mathcal{A}=\sum_{e\in \Vec{E}}w_eA_e$, we have
     $$\mathrm{det} \,\mathcal{A}=\sum_{\gamma \in \Gamma}(-1)^{\frac{|V|+|V_\gamma|}{2}+\mathrm{cc}(\gamma)}\left(\sum_{M\in \mathcal{PM}(G\setminus \gamma)}\prod_{e\in \vec{E}_M}w_e\right) w_{\gamma},$$
where $\mathrm{cc}(\gamma)$ is the number of connected components of $\gamma$.
\end{lemma}
     
\begin{proof} We expand the determinant as
    \begin{align*}
\mathrm{det}\,\mathcal{A}&=\sum_{\sigma\in S_n} {\mathrm{sgn}(\sigma)}\mathcal{A}_{1,\sigma(1)}\mathcal{A}_{2,\sigma(2)}\cdots \mathcal{A}_{n,\sigma(n)}\\
    &=\sum_{(e_1,e_2,\cdots,e_n)\in \vec{E}^n}\,\sum_{\sigma\in S_n} {\mathrm{sgn}(\sigma)}w_{e_1}w_{e_2}\cdots w_{e_n}(A_{e_1})_{1,\sigma(1)}(A_{e_2})_{2,\sigma(2)}\cdots (A_{e_n})_{n,\sigma(n)}.
    \end{align*}
Given any $(e_1,\ldots, e_n)\in \vec{E}^n$, there exists a  $\sigma \in S_n$ such that 
    \[
		(A_{e_1})_{1,\sigma(1)}(A_{e_2})_{2,\sigma(2)}\cdots (A_{e_n})_{n,\sigma(n)}\neq 0,
		\]
		if and only if \(o(e_i)=i\) and the termini of \(e_i\) are all distinct, which gives the permutation \(\sigma\) by mapping \(i\) to \(t(e_i)\).
        In this case, we obtain a decomposition $\{e_1,\ldots,e_n\}=\vec{E}_1 \sqcup \vec{E}_2$, where $\vec{E}_1$ is symmetric, i.e., $e\in \vec{E}_1$ implies that $e^{-1} \in \vec{E}_1$, and $\vec{E}_2=\vec{E}_\gamma$ for some oriented degree-2 subgraph $\gamma$ of $G$, and possibly one of the two sets is empty. To see this, note that in the (disjoint) cycle decomposition of $\sigma$, cycles of length $3$ or more correspond to oriented cycles in $G$, transpositions either correspond to a matching or a cycle of length $2$ if there are multi-edges,
        and the fixed points correspond to self-loops. Self-loops and cycles all contribute to $\vec{E}_2$, and the rest, the symmetric set $\vec{E}_1$, defines a matching on $G$, which is perfect on $G\setminus \gamma$, as $\sigma$ is a bijection.

Such a $\sigma \in S_n$ is unique if it exists and we have 
\[
 (A_{e_1})_{1,\sigma(1)}(A_{e_2})_{2,\sigma(2)}\cdots (A_{e_n})_{n,\sigma(n)}=1.
\]

Note that $\sigma$ is a composition of $\gamma$ (viewed as a permutation by mapping each vertex $i\in V_\gamma$ to the terminus of the directed edge originating from $i$ in \(\gamma\) and fixing the vertices in $V \setminus V_\gamma$) and $\frac{|V|-|V_\gamma|}{2}$ involutions from the matching. The signature of $\gamma$ satisfies
\[\mathrm{sgn}(\gamma)=(-1)^{|V_\gamma|-\mathrm{cc}(\gamma)},\]
where \(\mathrm{cc}(\gamma)\) is number of connected components of \(\gamma\). Therefore, $\mathrm{sgn}(\sigma)=(-1)^{\frac{|V|+|V_\gamma|}{2}+\mathrm{cc}(\gamma)}$.

Conversely, it is easy to see that for an oriented degree-$2$ subgraph $\gamma$ in $G$ and a perfect matching $M$ on $G \setminus \gamma$, there exists a unique tuple $(e_1,e_2,\cdots,e_n)\in \vec{E}^n$ such that $\{e_1,e_2,\cdots,e_n\}=\vec{E}_M \sqcup \vec{E}_{\gamma}$ and a permutation \(\sigma\) such that
    $$(A_{e_1})_{1,\sigma(1)}(A_{e_2})_{2,\sigma(2)}\cdots (A_{e_n})_{n,\sigma(n)}=1.$$
Put together, we obtain a bijection between the set of pairs $((e_1,\ldots,e_n),\sigma)\in \vec{E}^n\times S_n$ for which the summand is non-zero and pairs $(\gamma,M)\in \Gamma\times\mathcal{PM}(G\setminus\gamma)$, and the previous computations of the summand give precisely the lemma.
\end{proof}

The previous lemma can also be used to compute the determinant of $\mathcal{A}(z)$ by replacing the weights $w_e$ by $w_ez_e$ in the expansion. We use Lemma \ref{lemma:edgeweightexpansion} to expand the characteristic polynomial of the Schr\"{o}dinger operator on a finite multi-graph as follows.

\begin{proposition}\label{prop:expansionofschrodinger}
Let $G=(V,E)$ be a finite multi-graph. The characteristic polynomial of $\mathcal{H}$ on $G$ has the following expansion
 \begin{equation}
        \mathrm{det}\left(\lambda I-\mathcal{H}\right)=\sum_{\gamma \in \Gamma} (-1)^{\mathrm{cc} (\gamma)}m^{\mathcal{H}}_{G\setminus \gamma}(\lambda) w_{\gamma},
    \end{equation}
where $\mathrm{cc}(\gamma)$ is the number of connected components of $\gamma$. 
\end{proposition}
\begin{proof}
We make use of the following fact: Let $A$ be an $n\times n$ matrix, then for any $D=\mathrm{diag}\left(d_1,d_2,\cdots,d_n\right)$, we have 
$$\mathrm{det}(D+A)=\sum_{S\subset [1,n]} \left(\prod_{i\in [1,n]\setminus S}d_i\right) \mathrm{det}A_S,$$
where $\mathrm{det}\,A_S$ is the principal minor of $A$ corresponding to the index set $S$. 

Recall that for $S\subseteq V$, $G[S]$ denotes the subgraph of $G$ induced on the vertex set $S$. The proof is then a direct result of the fact above and Lemma \ref{lemma:edgeweightexpansion}:
     \begin{align*}
     &\mathrm{det}\left(\lambda I-\mathcal{H}\right)\\
     %\=&(-1)^{|V|}\cdot\mathrm{det} \left(\mathcal{V}+\mathcal{A}-\lambda I\right)\\
        =&(-1)^{|V|}\cdot\sum_{S\subset V} \left(\prod_{u\in V\setminus S} \left(\mathcal{V}_u-\lambda\right)\right) \mathrm{det}\mathcal{A}_S\\
    =&(-1)^{|V|}\cdot\sum_{S\subset V} \left(\prod_{u\in V\setminus S} \left(\mathcal{V}_u-\lambda\right)\right)\sum_{\gamma \in \Gamma_{G[S]}} (-1)^{\frac{|S|+ |V_\gamma|}{2}+\mathrm{cc}(\gamma)}\, w_{\gamma}\sum_{M\in \mathcal{PM}\left(G[S]\setminus \gamma\right)}\prod_{e\in \vec{E}_M}w_e  \\
        =&(-1)^{|V|}\cdot\sum_{\gamma \in \Gamma_{G}}w_{\gamma}  \sum_{S\subset V: V_\gamma \subset S } (-1)^{\frac{|S|- |V_\gamma|}{2}+ |V_\gamma|+\mathrm{cc}(\gamma)}\,\left(\prod_{u\in V\setminus S} \left(\mathcal{V}_u-\lambda\right)\right)\sum_{M\in \mathcal{PM}(G[S]\setminus \gamma)}\prod_{e\in \vec{E}_M}w_e\\
    =&(-1)^{|V|}\cdot\sum_{\gamma \in \Gamma_{G}}w_{\gamma}  \sum_{M\in \mathcal{M}\left(G\setminus \gamma \right)}(-1)^{\frac{|V_M|}{2}+|V_{\gamma}|+\mathrm{cc}(\gamma)} \left(\prod_{u\in V\setminus \left(V_\gamma \sqcup V_M \right)}\left(\mathcal{V}_u-\lambda\right)\right)  \prod_{e\in \Vec{E}_M}w_e\\
    =&\sum_{\gamma \in \Gamma_G} (-1)^{\mathrm{cc}( \gamma)}m^{\mathcal{H}}_{G\setminus \gamma}(\lambda)  w_{\gamma}.
    \end{align*}
\end{proof}

\begin{remark}
    From the Möbius inversion formula, we have that 
    \begin{equation}\label{eq:matchingexpansion}
        m^{\mathcal{H}}_{G}(\lambda)=\sum_{\gamma \in \Gamma} \mathrm{det}\left(\lambda I-\mathcal{H}|_{G\setminus \gamma}\right)  w_{\gamma}.
    \end{equation} 
\end{remark}

As with Lemma \ref{lemma:edgeweightexpansion}, we may replace the weights $w_e$ by $w_ez_e$ in the previous lemma to obtain the analogous expansion for $\mathrm{det}\left(\lambda I-\mathcal{H}(z)\right)$.

\subsection{Proof of the general criterion}
%The proof of Theorem \ref{thm:criteria} now immediately follows from the determinant expansion of Proposition \ref{prop:expansionofschrodinger} for the characteristic polynomial of $\mathcal{H}(z)$.
We are ready to combine the results of Sections~\ref{Sec:constructionofmaxabelian} and \ref{sec:matrev} to prove our first result.

\begin{proof}[Proof of Theorem \ref{thm:criteria}]
By Proposition \ref{prop:eigenvalue-correspondence}, $\lambda$ is an eigenvalue for $\mathcal{H}$ on $G^\mathrm{ab}$ if and only if it is an eigenvalue of $\mathcal{H}(z)$ for all $z\in \mathbb{T}^{|\vec{E}_0|}$. From Proposition \ref{prop:expansionofschrodinger}, we have
    \begin{align*}
        \mathrm{det} \left(\lambda I-\mathcal{H}(z)\right)=\sum_{\gamma \in \Gamma} (-1)^{\mathrm{cc}(\gamma)}m^{\mathcal{H}}_{G\setminus \gamma}(\lambda)  w_{\gamma} z_{\gamma},
    \end{align*}
where for a directed cycle $\gamma=\left(V_\gamma, \vec{E}_\gamma \right)$, we set $z_\gamma:=\prod_{e\in \vec{E}}z_{e}$ and $z_{e^{-1}}:=\overline{z_e}$ for all $e\in\vec{E}_0$. As functions in $L^2(\mathbb{T}^{|\vec{E}_0|})$, the $z_\gamma$ are mutually orthogonal; thus, $\mathrm{det} \left(\lambda I-\mathcal{H}(z)\right)\equiv 0$ in $L^2(\mathbb{T}^{|\vec{E}_0|})$ if and only if all of the coefficients of the $z_\gamma$ are zero. However, this holds if and only if $\lambda$ is a root of $m_{G\setminus\gamma}^\mathcal{H}(x)$ for every degree-$2$ subgraph $\gamma$ because $w_\gamma\neq 0$.
\end{proof}

\begin{remark}\label{remark:cha=mat}
    Combining Theorem \ref{thm:criteria} with  \eqref{eq:matchingexpansion}, we deduce that $\lambda$ is an eigenvalue of $\mathcal{H}^{\mathrm{ab}}$ if and only if it is an \emph{eigenvalue} of $\mathcal{H}|_{G\setminus \gamma}$ for any degree-$2$ subgraph $\gamma$, as mentioned in Section~\ref{Sec:Intro}.
\end{remark}

As shown in \cite[Lemma 3.6]{marcus2015interlacing}, see also \cite[Chapter 6, Theorems 1.1 and 1.2]{godsil1993algebraic}, the roots of the standard matching polynomial of a finite graph 
\(G\) are bounded in absolute value by the spectral radius of its universal cover. This result extends to generalized matching polynomials without difficulty. Thus a direct consequence of Theorem \ref{thm:criteria} is that all eigenvalues of $\mathcal{H}^{\mathrm{ab}}$, if they exist, lie within the Ramanujan bound of $G$ with respect to $\mathcal{H}$. More precisely, we have the following.

\begin{corollary}
\label{cor:ram-bound}
    Let $\lambda$ be an eigenvalue of $\mathcal{H}^{\mathrm{ab}}$, then
      $|\lambda|\leq \rho$ where $\rho$ is the spectral radius of $\mathcal{H}^{\mathrm{uni}}$, the pull-back of $\mathcal{H}$ to the universal cover of $G$.
\end{corollary}

\section{Maximal abelian covers of regular multi-graphs}\label{Sec:atomthm}

In this section we will prove Theorem \ref{thm:regular}. When the graph is even regular, the proof will be a straightforward consequence of the graph having a 2-factor, and so we will focus only on the case where the graph is odd regular. To this end, we begin by defining some notions regarding the structure of multi-graphs $G$.

\begin{definition}
Define the \textbf{leaf set} of $G$ as
    $$L_G:=\{v\in V_G: \mathrm{deg}_G(v)=1\}.$$
Each $v\in L_G$ is called a \textbf{leaf}.
\end{definition}

\begin{definition}
    A \textbf{bridge} in a multi-graph $G$ is an edge that disconnects the graph when removed and \(G\) is called bridge-less if no such edge exists. A \textbf{bridge-block} in $G$ is a connected component of the graph obtained after removing all of its bridges. Note that when removing a bridge we always keep the underlying vertices. 
\end{definition}

It is easy to see that no cycle in a graph can cross a bridge and conversely, every edge that is not a bridge is contained in a cycle. A bridge-block could be a singleton and this occurs precisely when all of the incident edges to a vertex are bridges.

\begin{lemma}
    The induced subgraph on each bridge-block is bridge-less.
\end{lemma}

\begin{proof}
    Suppose for a contradiction that there is an edge $e$ that is not a bridge in $G$ but it is a bridge in a bridge-block after removal of all bridges of $G$. Since $e$ is itself not a bridge of $G$, there is a cycle in \(G\) traversing it. As cycles cannot cross bridges, the cycle lives in the bridge-block, a contradiction.

    %its removal leaves $G$ connected, and so we can find a path in $G$ not crossing $e$ that joins its end points. Such a path must traverse bridges in $G$ otherwise $e$ cannot disconnect the induced subgraph on the bridge-block when removed since otherwise the path would reside entirely in the block. But adding $e$ to the path gives a cycle in $G$ passing over bridges, which contradicts the fact that they are bridges.
\end{proof}

\begin{definition}
    Define the \textbf{bridge-block tree} of a multi-graph $G$ to be the graph whose vertices are the bridge-blocks of $G$, and there is an edge between two vertices if and only if there is a bridge between the corresponding bridge-blocks in $G$.
\end{definition}

Note that every edge of the bridge-block tree is a bridge by construction, thus it can have no cycles and is a tree if the graph is connected, and a forest otherwise, see Figure \ref{bridgeblocktree}. 

\begin{figure}[h!]
\centering
\includegraphics[]{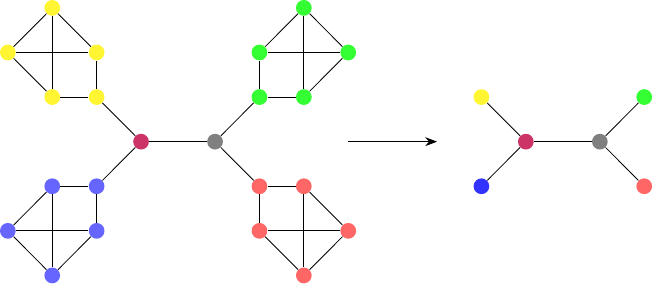}
\caption{Left: A cubic graph. Right: The associated bridge-block tree.}
\label{bridgeblocktree}
\end{figure}

Let \(G\) be a finite \(d\)-regular multi-graph with \(d\) odd.  
Each bridge-block of \(G\) falls into one of the following two categories:
\begin{itemize}
    \item every vertex of the block is incident to a bridge;
    \item at least one vertex of the block is not incident to a bridge.
\end{itemize}
A bridge-block is of the first type if and only if its maximum degree is strictly less than \(d\), and of the second type if and only if its maximum degree equals \(d\). For convenience, we adopt the following convention. 

\begin{definition}
    Fix an odd integer \(d\).  
    A finite multi-graph is \textbf{Type I} if its maximum degree is strictly less than \(d\), and \textbf{Type II} if its maximum degree equals \(d\) and is \emph{bridge-less}. 
\end{definition}

We call a bridge-block in \(G\) Type I (resp. Type II) if it is Type I (resp. Type II) as an induced subgraph of \(G\), see Figure \ref{fig:block}.

\begin{figure}[h!]

\centering
\includegraphics[]{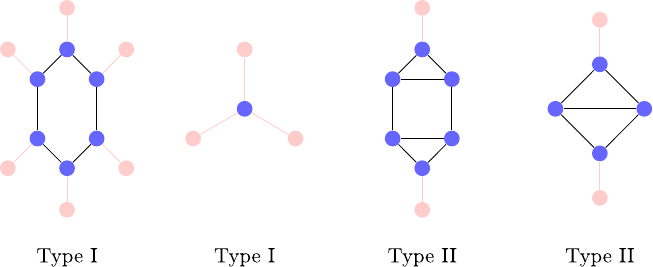}
\caption{Examples of bridge-blocks in a cubic graph of both types. Vertices in the blocks are colored blue whereas the bridges are colored pink.}
\label{fig:block}
\end{figure}

\subsection{Overview of the proof}
\label{sec:proof_overview}

To prove that the operator $\mathcal{H}^\mathrm{ab}$ has no eigenvalues, it suffices by Theorem \ref{thm:criteria} to show that for any $\lambda\in\mathbb{R}$, we can find a degree-2 subgraph $\gamma$ in $G$ such that $\lambda$ is not a root of the matching polynomial of $G\backslash\gamma$.

If $G$ has at most one bridge, it is known (Lemma~\ref{Lemma: keylemma}) that it has a 2-factor. Removing it leaves the empty graph, and Theorem~\ref{thm:regular} is immediate. Thus the most challenging case is where \(G\) contains at least one bridge-block that is not a leaf in the bridge-block tree.

\smallskip

\noindent\textbf{Step 0. Degree-\(2\) subgraphs in Type II graphs.}

The fundamental step in the proof is a mechanism for finding certain degree-2 subgraphs in Type II blocks. Precisely, let $B$ be a Type II block in $G$ and distinguish a vertex $v$ of degree smaller than $d$ in the induced graph on $B$. In Proposition \ref{prop: smoothprocedure} we show that there are two degree-2 subgraphs of $B$, one containing $v$ and one not, but both containing all vertices of degree $d$ in $B$. The proof uses the fact that odd-regular graphs without too many bridges contain a $2$-factor (Lemma \ref{Lemma: keylemma}) as well as a novel degree-correction tool that minimally alters the graph geometry. 

Applying Proposition \ref{prop: smoothprocedure} to leaf-blocks, our criterion immediately implies that for $\lambda$ to be an eigenvalue of $\mathcal{H}^\mathrm{ab}$, it must be a root of the matching polynomial of every graph obtained by replacing each leaf block in $G$ by a singleton or empty graph, followed by deleting any degree-2 subgraph. 

\smallskip

\noindent\textbf{Step 1. The star graph case and the induction step.}

From now on, we consider $G$ to have all leaf-blocks replaced by singletons. The first interesting case arises when the bridge-block tree is a star graph. If the central block is Type I, meaning that each vertex in this block has an attaching leaf in \(G\), we apply the recursion relation of matching polynomials to show that by simply deleting some leaves, \(\lambda\) is not a matching polynomial root of the remaining graph, see Lemma \ref{lemma:randomleafavoideigenvalue}. 

If however the central block \(B\) is Type II, then by Step 0, there exists some degree-\(2\) subgraph that covers all degree-\(d\) vertices. Upon the deletion of this degree-\(2\) subgraph, we obtain a Type I graph whose vertices are attached to leaves and we argue exactly as above and Theorem \ref{thm:regular} holds for this special case.

The general case, which we explain for the rest of this overview, requires a more complicated argument through an induction on the number of non-leaf blocks in $G$. That is, an induction on the following statement: If $G$ has $k$ non-leaf bridge blocks, then for any $\lambda\in\mathbb{R}$, we can delete a degree-2 subgraph and a collection of leaves from $G$ so that $\lambda$ is not a root of the generalized matching polynomial of the remaining graph. 

By a simple argument we can find a block $B$  for which all but one of its incident bridge blocks are leaves. Let $B'$ denote the induced subgraph of $G$ on $B$ and its attached leaves, and let $G'$ be the graph obtained from $G$ by replacing $B'$ by a leaf $\ell$. Since $G'$ has one less non-leaf block than $G$, the inductive hypothesis yields a degree-\(2\) subgraph $C$ and a collection of leaves $L$ in $G'$ for which $\lambda$ is not a root of the generalized matching polynomial of $G'\setminus(C\sqcup L)$.

\smallskip

\noindent\textbf{Step 2. Reintroduction of $B'$.}

We now wish to reintroduce $B'$. The way this is done depends upon if the leaf $\ell$ replacing $B'$ in $G'$ is in $L$ or not. Let $v$ be the vertex in $B$ incident to the non-leaf block in $G$. We assume that $B$ is a Type II block (Type I blocks are treated in an easier manner similar to Step 4 below) and $v$ is the distinguished vertex of $B$ as in Step 0.

\begin{enumerate}
	\item If $\ell\in L$, delete from $G$ the degree-2 subgraph of $B$ containing $v$ from Step 0. This disconnects what remains of $B'$ from $G$.
	
	\item If $\ell\notin L$, delete instead the degree-2 subgraph of $B$ not containing $v$.
\end{enumerate}

\smallskip

\noindent\textbf{Step 3. Destroy matching polynomial roots in components not containing $v$.}

After removing the degree-2 subgraph from Step 2 together with $C$ and $L\setminus\{\ell\}$ from Step 1, every connected component of $G$ not containing $v$ has one of the following types. 

\begin{enumerate}
	\item Contained in $G'$. By Step 1, deleting $C$ and $L\setminus\{\ell\}$ ensures that $\lambda$ is not a root of its generalized matching polynomial.
	
	\item Contained in $B'$. In this case every vertex in the component is either a leaf or adjacent to a leaf. By using the recursion relation of matching polynomials, we can find through Lemma~\ref{lemma:randomleafavoideigenvalue} a collection of vertices to delete in the connected component that are leaves in $G$ so that $\lambda$ is not a root of the generalized matching polynomial of the remaining graph.% -- see Lemma \ref{lemma:randomleafavoideigenvalue}.
\end{enumerate}

\noindent\textbf{Step 4. Destroy matching polynomial roots in the component containing $v$.}

The component containing $v$, which exists only if $\ell\notin L$, is a graph obtained by connecting two disjoint subgraphs $G_{\mathrm{co}}$ and $G_{\mathrm{at}}$ of $G$ by a bridge. Here $G_\mathrm{co}$ is the component of $G'$ that contained $\ell$ after deleting $L$ and $C$ but now with the leaf \(\ell\) removed, and $G_{\mathrm{at}}$ is the component in $B'$ containing $v$ after removal of the degree-2 subgraph in Step 2. For simplicity, assume that $G_\mathrm{co}$ is not the empty graph and $G_{\mathrm{at}}$ is not a singleton.

Every vertex except \(v\) in \(G_{\mathrm{at}}\) is either a leaf or adjacent to a leaf. As a consequence, by using the recursion relation of matching polynomials, we show the following (Lemma \ref{lemma:lemmaforcontradiction}): If $\lambda$ is a root of the generalized matching polynomial of any graph obtained from connecting $G_{\mathrm{co}}$ and $G_{\mathrm{at}}$ by reintroducing the bridge connecting them, and removing a collection of leaves from $G_{\mathrm{at}}$, then it is also a root of $G_{\mathrm{co}}$ with \(\ell\) reattached. But this contradicts the induction hypothesis unless there is a collection of leaves that can be deleted from the component such that $\lambda$ is not a root of the generalized matching polynomial.

The result follows since we have found a degree-2 subgraph and a collection of leaves (which correspond to a choice of degree-2 subgraph in leaf blocks) in $G$ whose deletion yields a graph for which the generalized matching polynomial of every connected component does not have $\lambda$ as a root.

\subsection{Properties of Type II graphs}\label{Section:TypeII}

We begin by proving the result about finding degree-$2$ subgraphs in Type II blocks that is used to treat leaf-blocks as leaves in Step 0 of the proof outline, and used to reduce Type II blocks to Type I graphs except possibly at a distinguished vertex. A key tool for this is the following result about finding 2-factors in regular multi-graphs \cite[Theorem 2.2]{kostochka2021cut}, see also \cite[Corollary 1]{Han98}.

\begin{lemma}\label{Lemma: keylemma}
    If $G$ is a $d$-regular multi-graph with at most $d-1$ bridges, then $G$ has a $2$-factor. 
\end{lemma}

Note that if $d$ is even, then it always has a $2$-factor and is bridge-less. We assume that $d$ is odd from now on. To make use of Lemma \ref{Lemma: keylemma} we will often need to correct the degree of certain vertices in blocks of the graph. One way to do this is by attaching graphs that have a single vertex of degree $d-1$ and all other vertices (if they exist) of degree $d$, by bridges. The simplest example is a bouquet with \((d-1)/2\) self-loops. 
This procedure leads to the following result.

\begin{lemma}\label{lemma:2factor}
Suppose that $G=(V,E)$ is a finite, connected bridge-less multi-graph with maximum degree at most $d$, with $d\geq 3$ odd. 
If $$\sum_{v\in V} (d-\mathrm{deg}_G(v))\leq d-1,$$
then $G$ has a \(2\)-factor.
\end{lemma}

\begin{proof}
We begin by joining copies of bouquets with $(d-1)/2$ self-loops to each vertex of degree less than \(d\) via a bridge edge to correct the degree deficit. The number of copies added at a vertex \(v\) is precisely $d-\mathrm{deg}(v)$ and so since $$\sum_{v\in V} (d-\mathrm{deg}_G(v))\leq d-1,$$ we obtain a $d$-regular graph with at most $d-1$ bridges. By Lemma \ref{Lemma: keylemma} this means the new graph has a 2 factor. Since cycles cannot cross bridges, it must be that the $2$-factor restricts to one on the original graph $G$.
\end{proof}

    If a Type II graph has only one vertex $o$ of degree less than $d$, we can use the previous lemma to show that there exists two degree-\(2\) subgraphs in the graph, with one containing all vertices and the other containing all vertices except for $o$, see Figure \ref{fig:twodegree2subgraphinhouse5}. Using this fact as a starting point, we show by induction that a similar phenomenon actually occurs in all Type II graphs, and allows us to construct degree-2 subgraphs whereby we can choose whether or not they contain a given distinguished vertex of degree less than $d$.

\begin{figure}[h!]
    \centering
\includegraphics[]{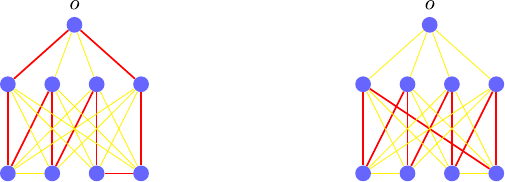}
    \caption{Two degree-$2$ subgraphs of a Type II graph such that the left contains $o$ and right does not, edges of these subgraphs are in red.}
    \label{fig:twodegree2subgraphinhouse5}
\end{figure}

\begin{proposition}
    \label{prop: smoothprocedure}

Let $G=(V,E)$ be a finite, connected, bridge-less multi-graph with maximum degree $d$ with $d\geq 3$ odd. Let 
$$I=I(G):=\{v \in V: \mathrm{deg}_G(v)<d \}.$$ 
For any $v\in I$, there exists two degree-$2$ subgraphs in $G$ one of which has vertex set containing $v$ and one that does not. In both cases, the vertex sets contain at least $G\setminus I$.
\end{proposition} 

\begin{proof}
    We proceed by induction on $|I|$. For the case $|I|=1$, let $v \in I$. The existence of the first degree-$2$ subgraph that contains $v$ then follows directly from Lemma \ref{lemma:2factor}.

    For the second collection, delete $v$ (and all edges incident to it). Even though $G$ is bridge-less, the deletion of $v$ may cause bridges to appear. %The vertex $v$ has at least two edges incident to it that are not self-loops, otherwise there will be a bridge in $G$. Let $2 \leq n_1 \leq d-1$ be the number of such edges and $n_2=d-\mathrm{deg}(v)$. 
    
     Recall that if $v_0$ is a vertex in a forest such that $\mathrm{deg}(v_0)=m$, then there exists at least $m$ distinct leaves in the forest (and also distinct from $v_0$, in case it is itself a leaf). Now consider the bridge-block tree of $G$ with $v$ removed. Given a block $B$ that is of degree $m$ in the bridge-block tree, we can find at least $m$ different blocks that are leaves in the bridge-block tree (that are also distinct from $B$) by the previous observation. However, if a block is a leaf in the bridge-block tree, then there has to be an edge between it and the vertex $v$ in the original graph otherwise $G$ would have had a bridge. 
    
    Suppose then that in $G$, $B$ is connected to $v$ by $a$ edges, then we have $$m+a\leq d-1.$$ Indeed, since each of the blocks distinct from $B$ (for which there are at least $m$) is attached to $v$ by at least one edge in the original graph, and $B$ itself is connected to $v$ by exactly $a$ edges, we have that $m+a$ is bounded above by the degree of $v$ in $G$ which is at most $d-1$.
    
    Thus, each block in the bridge-block tree obtained after deleting $v$ satisfies the condition of Lemma \ref{lemma:2factor} since 
    \[\sum_{v\in V_B}\left(d-\mathrm{deg}_B(v)\right)=m+a\leq d-1,\]
    and so we can find a degree-$2$ subgraph containing all of the vertices in a given block. Combining the degree-\(2\) subgraphs for each block gives a degree-$2$ subgraph in $G$ containing all vertices except $v$ as required.

Now suppose that the lemma holds for $|I(G)|=k$ for some $k\geq 1$. For the case where $|I(G)|=k+1$, distinguish the vertex $v$ and pick $v_1\in I, v_1\neq v$. We do the following degree-correction procedure to create a new graph $G'$ such that $I(G')=k$ (see Figure \ref{fig:oddsmoothing}):
\begin{itemize}
    \item If $\mathrm{deg}(v_1)$ is odd, we attach to it $\frac{d-\mathrm{deg}(v_1)}{2}$ self-loops.
    \item If $\mathrm{deg}(v_1)$ is even, let $u\in G$ be one of its neighbors that is not \(v_1\) itself. Now, remove the edge joining $v_1$ and $u$ (if there are multiple such edges, we just remove one of them). Next, we introduce a new vertex $v_1'$ and add $d+1-\mathrm{deg}(v_1)$ edges between $v_1$ and $v_1'$. Then, we put an edge between $v_1'$ and $u$, and attach $\frac{\mathrm{deg}(v_1)-2}{2}$ self-loops to $v_1'$ .
\end{itemize}

\begin{figure}[h!]
    \centering
   \includegraphics[]{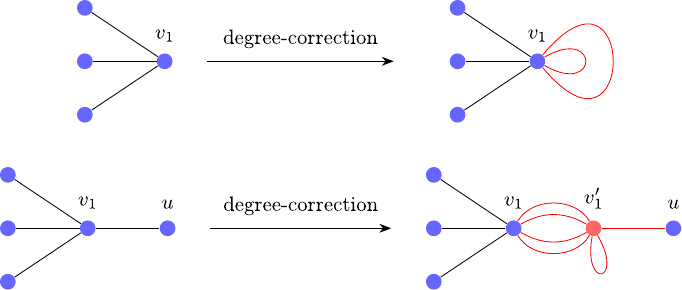}
\caption{The degree-correction procedure of degree $d=7$. Top:  $\mathrm{deg}(v_1)=3$, bottom: $\mathrm{deg}(v_1)=4$.}
    \label{fig:oddsmoothing}
\end{figure}
In the resulting graph $G'$, the degree of $v_1$ and the new vertex $v_1'$ is $d$ whereas the other vertices have their degrees unchanged and so $|I(G')|=k$, and it is straightforward to check that $G'$ is also bridge-less. By induction, we can find two degree-$2$ subgraphs in $G'$, one that contains the vertex $v$ and one that does not, but in both cases, they contain the vertex set $V(G')\setminus I(G')$. Let us construct from each a degree-$2$ subgraph in the original graph $G$ that satisfies the conclusion of the proposition.

\smallskip

\textbf{Case 1}: Suppose that $\mathrm{deg}(v_1)$ is odd, then if the degree-$2$ subgraph does not contain one of the newly added self-loops, it is also a degree-$2$ subgraph in $G$ and we are done. If however it does contain a newly added self-loop, then after we delete the self-loop, what is left is a degree-$2$ subgraph in $G$. The only vertex that is not contained in its vertex set compared to the previous subgraph in $G'$ is possibly $v_1$. This is permitted since $v_1 \in I$.

\smallskip

\textbf{Case 2}: Suppose that $\mathrm{deg}(v_1)$ is even, we have three sub-cases depending on how $v_1'$ is contained in the degree-$2$ subgraph:
\begin{itemize}
    \item The degree-$2$ subgraph in $G'$ contains a newly added self-loop at $v_1'$. Then we know that without this self-loop, the remainder of the subgraph does not cross the edges joining $v_1$ to $v_1'$ or $u$ to $v_1'$. Thus by deleting the self-loop from the subgraph, what remains is a degree-$2$ subgraph in $G$. In this case, the subgraph collection still contains the vertex $v_1$.
    
    \item The degree-$2$ subgraph in $G'$ contains a cycle that traverses two distinct edges between $v_1$ and $v_1'$. By deleting this cycle, what remains is a degree-$2$ subgraph in $G$ since no other cycle could cross edges emanating from $v_1'$. Note that this time $v_1$ is not in the vertex set of the subgraph of $G$.
    
    \item The degree-$2$ subgraph crosses only one new edge between $v_1$ and $v_1'$ and hence also the edge joining $v_1'$ to $u$. By contracting $v_1$ and $v_1'$ back to one vertex, we get a degree-$2$ subgraph in $G$ containing a cycle that traverses the edge between $v_1$ and $u$. Note that in this case, $v_1$ is in the degree-\(2\) subgraph of $G$.    
\end{itemize}
In any case, we have found the two desired degree-$2$ subgraphs.
\end{proof}

\subsection{Properties of Type I graphs with attached leaves}\label{section:Type I}

The previous section essentially allows us to reduce Type II blocks to Type I graphs, potentially with an additional distinguished vertex. In the final proof, we will reduce to the case where every bridge attached to a Type I subgraph under consideration connects to a leaf. The next result shows that for such graphs and $\lambda\in\mathbb{R}$, we can always remove a subset of these leaves to ensure that $\lambda$ is not a root of the generalized matching polynomial of the remaining graph. 

\begin{lemma}\label{lemma:randomleafavoideigenvalue}
    Let $G=(V,E)$ be a finite, connected multi-graph with $|V| \geq 2$, a distinguished root vertex $o$, and leaf set $L$.  
    Suppose that $G$ satisfies the following property: for every $v \in V \setminus \{o\}$, either $v \in L$ or there exists $l_v \in L \setminus \{o\}$ adjacent to $v$.  
    Then, for any $\lambda \in \mathbb{R}$, there exists $L_0 \subset L \setminus \{o\}$ such that $\lambda$ is not a root of the generalized matching polynomial $m_{G \setminus L_0}^{\mathcal{H}}(x)$.
\end{lemma}

\begin{proof}
First note that since self-loops do not contribute to a graphs matching polynomial and we will only be deleting leaves, we may assume that $G$ has no self-loops. Our proof is based on induction on $|V\setminus (L\cup \{o\})|$.

Suppose that $|V\setminus (L\cup \{o\})|=0$ so that since $|V|\geq 2$, the graph consists of the root $o$ with at least one leaf emanating from it. If we delete all but one of the leaves, say \(l\), attached to $o$, what remains is a graph with two vertices and one edge connecting them. Suppose that the edge weight is $w$ or $\bar{w}$ depending on the direction. The generalized polynomial of this graph then equals \(\left(x-\mathcal{V}_o\right)\left(x-\mathcal{V}_l\right)-|w|^2\). If we delete all leaves attached to $o$, then we get a singleton whose only matching polynomial root is $\mathcal{V}_o$. Our claim then trivially holds. 
    
Next consider $|V\setminus (L\cup \{o\})| =1$. Let $u\in V, u\neq o$ be the non-leaf vertex necessarily connected to $o$ since the graph is connected. If we delete all leaves attached to $o$ and $u$, what remains is the induced subgraph on $\{o,u\}$ consisting of $k$ edges joining them. We call this graph $G_0$, and denote the weights on the $2k$ directed edges are $w_1, \overline{w_1}, \cdots, w_k, \overline{w_k}$. If instead we delete all leaves attached to \(o\) and all but one leaf attached to \(u\) in the graph, what remains is the same graph but with a leaf, say $l$, attached to $u$. We denote this graph as $G_0\sqcup \{l\}$ and assume that the weight of the new edge is $w_0$ or $\overline{w_0}$ depending on its orientation.
If the matching polynomial of these two share the same root $\lambda$, then 
we have 
\[
m_{G_0}^{\mathcal{H}}(\lambda)=\left(\lambda-\mathcal{V}_o\right)\left(\lambda-\mathcal{V}_u\right)-\sum_{i=1}^k |w_i|^2=0,
\]
and also
\[
m_{G_0\sqcup \{l\}}^{\mathcal{H}}(\lambda)=\left(\lambda-\mathcal{V}_l\right)\cdot m_{G_0}^{\mathcal{H}}(\lambda)-|w_0|^2 \left(\lambda-\mathcal{V}_o\right)=0.
\]
Here we use the recursion relation of matching polynomials.
This is impossible since substituting the first equation into the second yields that $\lambda=\mathcal{V}_o$, which violates the first equation. So again the lemma holds.

Now assume that the lemma holds for all $|V\setminus (L\cup \{o\})|\leq k$, for some $k\geq 1$; we will show that it also holds for $|V\setminus (L\cup \{o\})|=k+1$ by contradiction. Suppose that $\lambda$ is a root of $m^{\mathcal{H}}_{G\setminus L_0}$ for all $L_0\subset L\setminus\{o\}$. Fix any $u\in V\setminus \left(L\cup\{o\}\right)$ that is connected to $o$ which necessarily exists since the graph is connected. Pick any $u_1\in V \setminus (L\cup\{o,u\}) $ and let $l_1, l_2, \cdots l_n$ be the leaves that are attached to $u_1$.

We now consider each connected component of the graph $G\setminus \{u_1, l_1, l_2, \cdots ,l_n\}$ and show that the inductive hypothesis applies to them. If the component contains $o$, then pick $o$ as the root vertex of that component, otherwise pick one of the former neighbors of $u_1$ to be the root. Then each component $C$ with root $r_C$ satisfies the assumptions of the lemma with $|V_C \setminus (L_C\cup\{r_C\})|\leq k$ and $L_C\setminus\{r_C\}\subseteq L$, see Figure~\ref{fig:randomleafavoideigenvalue}. We recall that $V_C$ and $L_C$ represent the vertex set and leaf set of $C$.

\begin{figure}[htbp]
\centering
\includegraphics[]{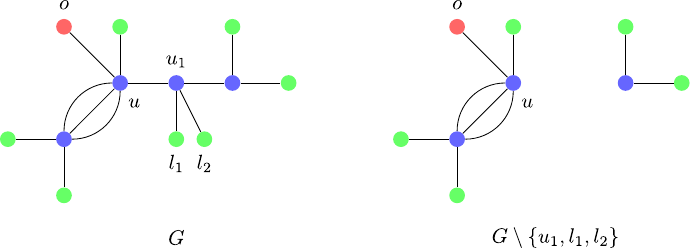}
\caption{Illustration for the proof of Lemma \ref{lemma:randomleafavoideigenvalue} with leaves colored green except for the distinguished vertex $o$.}
   \label{fig:randomleafavoideigenvalue}
   \end{figure}

Indeed, upon deletion of $u_1$ and the leaves $l_1,\ldots, l_n$, the only possible new leaves that have been created in a connected component $C$ that weren't leaves in the original graph are possibly vertices in $C$ that were connected to $u_1$ in the original graph. If such a vertex does become a leaf, then by the assumption in the lemma on the original graph, it must be connected to some leaf in the original graph, and nothing more. This means that the connected component has precisely two vertices joined by a single edge, one of which was a leaf in the original graph and thus the component trivially satisfies the assumption with $|V_C \setminus (L_C\cup\{r_C\})|=0$ and $L_C\setminus\{r_C\}\subseteq L$. 

In the other components, there are no leaves that weren't leaves in the original graph thus $L_C 
\subseteq L$. Moreover, upon the deletion of \(u_1\) we indeed have $|V_C \setminus (L_C \cup\{r_C\})|\leq k$. Additionally, any non-leaf vertex in the component that is not the root was also a non-leaf vertex in the original graph and hence connected to some leaf which remains a leaf in the component $C$ since we only deleted $u_1,l_1,\ldots,l_n$ none of which are leaves that were connected to vertices other than $u_1$. 

We can thus apply the inductive step to each component and find 
$$L'_0\subset\bigcup_C \left(L_C\setminus\{r_C\}\right)\subseteq L\setminus\{o\},$$
such that the matching polynomial of $G\setminus \left(\{u_1, l_1, l_2, \cdots ,l_n\} \sqcup L'_0\right)$ does not have $\lambda$ as a root. However, due to our assumption towards contradiction, we must have 
$$m_{G\setminus \left(\{l_1, l_2, \cdots ,l_n\} \sqcup L'_0\right)}^{\mathcal{H}}(\lambda)=m_{G\setminus \left(\{l_2, \cdots ,l_n\} \sqcup L'_0\right)}^\mathcal{H}(\lambda)=0.$$
The recursion relation of matching polynomials Lemma \ref{lem:recursion} then implies that
\begin{align*}
    |w_{\{u,l_1\}}|^2 \cdot m_{G\setminus \left(\{u_1, l_1, l_2, \cdots ,l_n\} \sqcup L'_0\right)}^\mathcal{H}&(\lambda)\\
    =-m_{G\setminus \left(\{l_2, \cdots ,l_n\} \sqcup L'_0\right)}^\mathcal{H}&(\lambda)+\left(\lambda-\mathcal{V}_{l_1}\right)\cdot m_{G\setminus \left(\{l_1, l_2, \cdots ,l_n\} \sqcup L'_0\right)}^\mathcal{H}(\lambda)=0,
\end{align*}
a contradiction. 
\end{proof}

We now study the matching polynomial of graphs obtained by attaching a graph from Lemma~\ref{lemma:randomleafavoideigenvalue} to another graph via a bridge from the root. This appeared in the proof outline in Step 4.

To this end, let $G_{\mathrm{co}}=(V_{\mathrm{co}},E_{\mathrm{co}})$ be any finite connected multi-graph with a distinguished vertex $o$ and Schr\"{o}dinger operator $\mathcal{H}_{\mathrm{co}}$. Suppose that $G_{\mathrm{at}}=(V_{\mathrm{at}},E_{\mathrm{at}}), |V_\mathrm{at}|\geq 2$ satisfies the assumptions of Lemma \ref{lemma:randomleafavoideigenvalue} equipped with a Schr\"{o}dinger operator $\mathcal{H}_{\mathrm{at}}$. That is, $G_{\mathrm{at}}$ is a finite connected multi-graph with a distinguished vertex $o_{\mathrm{at}}$ and leaf set $L$, such that for any $v\in V_\mathrm{at}\setminus\{o_{\mathrm{at}}\}$, either $v\in L$ or there exists $l_v\in L\setminus\{o_{\mathrm{at}}\}$ adjacent to $v$. 

By \textbf{attaching $G_\mathrm{at}$ to  $G_\mathrm{co}$}, we mean a graph $G_{\mathrm{jo}}=(V_{\mathrm{jo}}, E_{\mathrm{jo}})$ obtained by joining $o$ and $o_{\mathrm{at}}$ by a (bridge) edge. More precisely, we define $V_\mathrm{jo}=V_{\mathrm{co}} \sqcup V_{\mathrm{at}}$ and $E_\mathrm{jo}=E_{\mathrm{co}} \sqcup E_{\mathrm{at}} \sqcup \{(o,o_{\mathrm{at}})\}$ - see Figure \ref{fig:attaching}. The Schr\"{o}dinger operator $\mathcal{H}_{\mathrm{jo}}$ on $G_\mathrm{jo}$ is defined by putting any non-zero weight on the newly added bridge and ensuring that its restrictions to $G_\mathrm{co}$ and $G_\mathrm{at}$ are $\mathcal{H}_{\mathrm{co}}$ and $\mathcal{H}_{\mathrm{at}}$ respectively. Later, the edge weight on the bridge will arise from an ambient Schr\"{o}dinger operator on the whole graph within which $G_\mathrm{jo}$ is a subgraph. 

\begin{figure}[h!]
\centering
\includegraphics[]{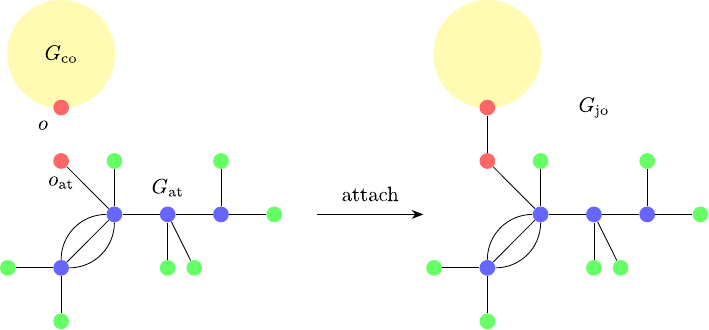}
\caption{Attaching  $G_\mathrm{at}$ to $G_\mathrm{co}$ by a bridge.}
\label{fig:attaching}
\end{figure}

\begin{lemma}\label{lemma:lemmaforcontradiction}
   If $\lambda \in \mathbb{R}$ is a root of the generalized matching polynomial of $G_{\mathrm{jo}} \setminus L_0$ for every $L_0 \subset L \setminus \{o_{\mathrm{at}}\}$,  
then $\lambda$ is also a root of the generalized matching polynomial of $G_{\mathrm{co}}$ and of $G_{\mathrm{co}}$ with an additional leaf $o_{\mathrm{at}}$ attached to $o$, in both cases with the Schrödinger operator $\mathcal{H}_{\mathrm{jo}}$ restricted to the respective graph.

\end{lemma}

Thus, if $\lambda \in \mathbb{R}$ is not a root of the generalized matching polynomials of one of $G_{\mathrm{co}}$ or $G_{\mathrm{co}}$ with an extra leaf attached to $o$, then there exists $L_0 \subset L\setminus \{o_\mathrm{at}\}$, such that $\lambda$ is not a root of the generalized matching polynomial of $G_{\mathrm{jo}}\setminus L_0$.

\begin{proof}
The proof is similar to Lemma~\ref{lemma:randomleafavoideigenvalue}, proceeding by induction on $|V_{\mathrm{at}}\setminus (L\cup \{o_\mathrm{at}\})|$. Again without loss of generality, we assume that all graphs considered here have no self-loops. For brevity, we write \(\mathcal{H}=\mathcal{H}_{\mathrm{jo}}\). 

First, if $|V_{\mathrm{at}}\setminus (L\cup \{o_\mathrm{at}\})|=0$, then $\lambda$ is a root of the generalized matching polynomials of $G_\mathrm{jo}$ after deleting all leaves attached to $o_\mathrm{at}$ or all but one leaf, say $l$, attached to $o_\mathrm{at}$ (there is at least one leaf since $|V_\mathrm{at}|\geq 2$). Therefore, we have 
\[
m_{G_\mathrm{co}\sqcup \{o_\mathrm{at}\}}^\mathcal{H}(\lambda)=m_{G_\mathrm{co}\sqcup \{o_\mathrm{at},l\}}^\mathcal{H}(\lambda)=0,
\]
where we see $G_\mathrm{co}\sqcup \{o_\mathrm{at}\}$ and $G_\mathrm{co}\sqcup \{o_\mathrm{at},l\}$ as induced subgraphs of $G_\mathrm{jo}$. Then the conclusion follows by the recursion Lemma~\ref{lem:recursion} since 
\[
\left|w_{\{o_{\mathrm{at}},l\}}\right|^2\cdot m_{G_\mathrm{co}}^\mathcal{H}(\lambda)=-m_{G_\mathrm{co}\sqcup \{o_\mathrm{at},l\}}^\mathcal{H}(\lambda)+\left(\lambda-\mathcal{V}_l\right) \cdot m_{G_\mathrm{co}\sqcup \{o_\mathrm{at}\}}^\mathcal{H}(\lambda)=0.
\]

Next, if $|V_\mathrm{at}\setminus (L\cup \{o_\mathrm{at}\})|=1$, let $u \in V_\mathrm{at}\setminus (L\cup \{o_\mathrm{at}\})$. Denote the subgraph of $G_{\mathrm{jo}}$ by deleting all leaves attached to $o_{\mathrm{at}}$ and all but one leaf attached to $u$, say $l_u$, as $G_{\mathrm{int}}$. By assumption, since we only deleted leaves in $L$, one must have that $m_{G_{\mathrm{int}}}(\lambda)=m_{G_{\mathrm{int}}\setminus \{l_u\}}(\lambda)=0.$
By the recursion relation for matching polynomials Lemma \ref{lem:recursion} we get 
    $$m_{G_{\mathrm{int}}}^\mathcal{H}(\lambda)=\left(\lambda-\mathcal{V}_{l_u}\right) \cdot m_{G_{\mathrm{int}}\setminus \{l_u\}}^\mathcal{H}(\lambda)-\left|w_{\{u,l_u\}}\right|^2\cdot m_{G_{\mathrm{int}}\setminus \{l_u,u\}}^\mathcal{H}(\lambda),$$
so that $m_{G_{\mathrm{int}}\setminus \{l_u,u\}}^\mathcal{H}(\lambda)=0$. Then, applying the recursion relation once more, we get 
    $$m_{G_{\mathrm{int}}\setminus \{l_u\}}^\mathcal{H}(\lambda)=\left(\lambda-\mathcal{V}_u\right)\cdot m_{G_{\mathrm{int}}\setminus \{l_u,u\}}^\mathcal{H}(\lambda)-\left(\sum_{e\in \vec{E};\, o(e)=u,\,t(e)=o_\mathrm{at}}|w_{e}|^2\right)\cdot m_{G_{\mathrm{int}}\setminus \{l_u,u,o_{\mathrm{at}}\}}^\mathcal{H}(\lambda).$$
It follows then that $m_{G_{\mathrm{int}}\setminus \{l_u,u,o_\mathrm{at}\}}^\mathcal{H}(\lambda)=0$. But, $G_{\mathrm{int}}\setminus \{l_u,u\}$ is exactly $G_{\mathrm{co}}$ with an extra leaf attached to $o$ and $G_{\mathrm{int}}\setminus \{l_u,u,o_\mathrm{at}\}=G_\mathrm{co}$ as required.

Now assume the lemma holds for $|V_{\mathrm{at}}\setminus (L\cup \{o_\mathrm{at}\})|\leq k$, with $k\geq 1$. For $|V_{\mathrm{at}}\setminus (L\cup \{o_\mathrm{at}\})|=k+1$, fix a $u\in V_{\mathrm{at}}\setminus (L\cup \{o_\mathrm{at}\})$ that is adjacent 
to $o_{\mathrm{at}}$. Then, choose $u_1 \in V_\mathrm{at}\setminus (L\cup\{u, o_{\mathrm{at}}\})$ and let
$\{l_1,l_2,\cdots,l_n\} \subset L$ be the leaves attached to $u_1$. For each component of the intermediate graph $G_{\mathrm{int}}:=G_{\mathrm{jo}}\setminus \{u_1,l_1,l_2,\cdots,l_n\}$, either it contains $o_\mathrm{at}, u$ and all vertices from $G_{\mathrm{co}}$, which we call $G_{\mathrm{int}}'$, or it satisfies the assumption of Lemma~\ref{lemma:randomleafavoideigenvalue} after picking a root for the component identically as in the proof of Lemma~\ref{lemma:randomleafavoideigenvalue}, see Figure \ref{fig:lemmaforcontradiction}.

\begin{figure}[h!]
\centering
\includegraphics[]{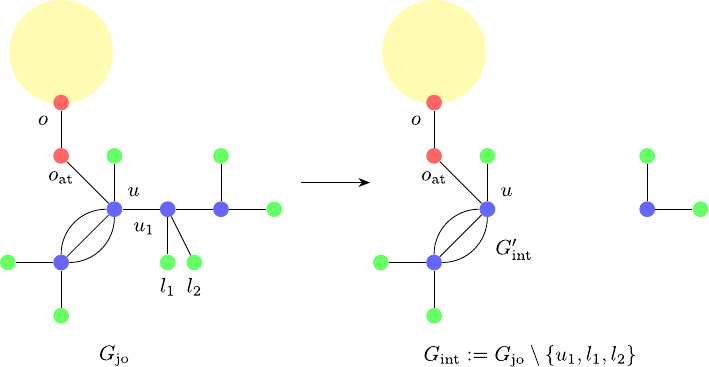}
\caption{Illustration for the proof of Lemma \ref{lemma:lemmaforcontradiction}.}
\label{fig:lemmaforcontradiction}
   \end{figure}

Now, decompose $L\setminus \{o_{\mathrm{at}}\}$ as $L\setminus \{o_{\mathrm{at}}\}=L_1\sqcup L_2 \sqcup \{l_1, l_2, \cdots l_n\}$, such that $L_1 \subset L$ are the leaves in $L\setminus \{o_{\mathrm{at}}\}$ that are in the connected component $G_{\mathrm{int}}'$. In particular, $L_1\neq\emptyset$ as $u$ has leaves attached.  Using Lemma \ref{lemma:randomleafavoideigenvalue} for those components that do not contain $o_{\mathrm{at}}$, there exists $L_2' \subset L_2$ such that after deleting $L_2'$, the generalized matching polynomial of the union of these components does not have $\lambda$ as a root. 

From our assumption, for any $L_1'\subset L_1$, the matching polynomial of $G_{\mathrm{jo}}\setminus (L_1' \sqcup L_2' \sqcup \{u_1,l_1,l_2,\cdots,l_n\})$ has root $\lambda$. Indeed, $n\geq 1$ and so by the recursion Lemma~\ref{lem:recursion} we have 
    \begin{align*}
    m_{G_\mathrm{jo}\setminus(\{\ell_2,\ldots,\ell_n\}\sqcup L_2'\sqcup L_1')}^\mathcal{H}(\lambda)&\\
    &\hspace{-3.5cm}=\left(\lambda -\mathcal{V}_{l_1}\right)\cdot m_{G_\mathrm{jo}\setminus(\{\ell_1,\ell_2,\ldots,\ell_n\}\sqcup L_2'\sqcup L_1')}^\mathcal{H}(\lambda)-\left|w_{\{u_1,l_1\}}\right|^2\cdot m_{G_\mathrm{jo}\setminus(\{u_1,\ell_1,\ell_2,\ldots,\ell_n\}\sqcup L_2'\sqcup L_1')}^\mathcal{H}(\lambda),
    \end{align*}
but by assumption, $m_{G_\mathrm{jo}\setminus(\{\ell_2,\ldots,\ell_n\}\sqcup L_2'\sqcup L_1')}^\mathcal{H}(\lambda)=m_{G_\mathrm{jo}\setminus(\{\ell_1,\ell_2,\ldots,\ell_n\}\sqcup L_2'\sqcup L_1')}^\mathcal{H}(\lambda)=0$. 

Since the matching polynomial of a disconnected graph is the product of the matching polynomials of its components, we obtain that the generalized matching polynomial of $G_{\mathrm{int}}'\setminus L_1'$ has root $\lambda$ for any choice of $L_1'\subseteq L_1$ due to the choice of $L_2'$. But $G_{\mathrm{int}}'$ can be constructed by attaching the connected component of $G_\mathrm{at}\setminus \{u_1\}$ containing $o_{\mathrm{at}},u$ to $G_{\mathrm{co}}$. Moreover, this connected component certainly has strictly less than $k+1$ vertices after removing its root and all of its leaves. It follows that we can use the inductive step on this new attaching graph to obtain the desired conclusion for $|V_{\mathrm{at}}\setminus \left(L \cup \{o_{\mathrm{at}}\}\right)|=k+1$.
\end{proof}

\subsection{Proof of Theorem \ref{thm:regular}}
\label{sec:prooreg}

We now have all the tools we need to prove Theorem~\ref{thm:regular}. Recall that in the proof outline, we considered a graph where all leaf-blocks were replaced by singletons. The next lemma shows that for this graph and any $\lambda\in\mathbb{R}$, we can indeed delete a degree-$2$ subgraph as well as some leaves to ensure that $\lambda$ is not a root of the generalized matching polynomial of the remaining graph.

\begin{lemma}\label{lemma: mainlemma}
    Let $G$ be a finite, connected multi-graph with Schr\"{o}dinger operator $\mathcal{H}$ and a non-empty leaf set \(L(G)\). Suppose that 
    \begin{itemize}
        \item for all \(v\in V_G\setminus L(G)\), \(\mathrm{deg}_G(v)=d\), with \(d\geq3\) odd;
        \item all leaf blocks of
        \(G\) are singletons.
    \end{itemize}
    Then for any $\lambda \in \mathbb{R}$, there exists $L_0\subset L$ and a degree-$2$ subgraph $C$ in $G$ such that $\lambda$ is not a root of the generalized matching polynomial of $G\setminus (L_0 \sqcup C)$. 
\end{lemma}

\begin{proof}
    We proceed by induction on the number of bridge-blocks $N(G)$ that are not leaf blocks. If $N(G)=0$, then $G$ is a graph with \(2\) vertices and an edge connecting them, in which case we can simply delete all of the vertices (as they are leaves) to obtain the empty graph whose generalized matching polynomial has no roots.
    
    Suppose next that $N(G)=1$ so that $G$ consists of a single bridge-block $B$ with leaves attached to it. If $B$ is a Type I block, then choosing any vertex in the block to be distinguished, the graph satisfies the hypothesis of Lemma \ref{lemma:randomleafavoideigenvalue} as every vertex in $B$ has attaching leaves from which the conclusion immediately follows.
    
    If $B$ is instead Type II, by Proposition \ref{prop: smoothprocedure}, we can find a degree-$2$ subgraph $C$ of $B$ containing all vertices except possibly some of those that are incident to bridges in $G$. Upon deletion of $C$ from $G$ along with any leaves that were incident to vertices in $C$, the remaining graph consists of components all of which (after arbitrarily choosing a distinguished vertex) satisfy the assumptions of Lemma \ref{lemma:randomleafavoideigenvalue} which upon application gives the desired conclusion. 

    Now assume that the lemma holds for all $G$ with $N(G) \leq k$ for some $k\geq 1$. For $N(G)=k+1$, we can always find a bridge-block $B$ (that is not a leaf in $L$) such that in the bridge-block graph it is connected to exactly one non-leaf vertex. Denote by $G_0$ the connected subgraph of $G$ obtained by deleting $B$ and all leaves attached to $B$.

    To see that such a $B$ exists, start at any non-leaf block in the bridge-block tree. If it has only one non-leaf neighbor then we are done, so assume this is not the case, and traverse one of the bridges to a non-leaf block. This new bridge-block either has the desired property or also has a new outgoing bridge not attached to a leaf-block. In the latter case, we again traverse this bridge, and continue in this manner. Note that we never return to a bridge-block that we have already visited since we only ever traverse new edges in the bridge-block tree and so returning to a previously visited block would mean there is a cycle in the tree. Thus, since there are only finitely many non-leaf bridge-blocks, the process must terminate at some point, and when it does, we have found a desired bridge-block $B$.
    
    Suppose now that $\lambda \in \mathbb{R}$ violates the conclusion of the lemma. Let $\{o,o'\}$ be the bridge separating $B$ and $G_0$, with $o'\in B$ and $o\in G_0$. Then the induced subgraph on $G_0 \sqcup \{o'\}$ satisfies our assumption in this lemma, with $N(G_0 \sqcup \{o'\})=k$. Therefore, we can find a degree-$2$ subgraph $C$ of $G_0 \sqcup \{o'\}$ (thus also a degree-$2$ subgraph of $G_0$ and $G$), and a set $L'$ of leaves in $G_0 \sqcup \{o'\}$, such that the generalized matching polynomial of $G_0 \sqcup \{o'\} \setminus\left( L' \sqcup C\right)$ does not have $\lambda$ as a root.
    
    We will now use $C$ and $L'$ to construct a new degree-$2$ subgraph and collection of leaves in $G$ whose deletion necessarily does not have $\lambda$ as a root of the generalizing matching polynomial to give a contradiction. There are several cases to consider depending on whether or not $o$ is a vertex in $C$.
    ~\newline
    
    \noindent \textbf{Case 1:} 
    $o$ is a vertex of $C$. Then after deleting $L'\setminus \{o'\}$ and $C$ from \(G\), all components except for the one containing $B$ and its attaching leaves do not have $\lambda$ as a root of their generalized matching polynomial. We show now that the component containing $B$ also does not have $\lambda$ as a root of its matching polynomial after possibly deleting a degree-$2$ subgraph and some leaves, giving the desired contradiction. There are two subcases:
    
\begin{itemize}
\item 
\textbf{(1a)} If $B$ is Type I, then $B$ and its attaching leaves satisfy the condition of Lemma~\ref{lemma:randomleafavoideigenvalue} when taking $o'$ as the distinguished root. We can then delete some leaves (that are not $o'$) such that the generalized matching polynomial of what remains does not have $\lambda$ as a root. See Figure~\ref{fig:1a} for an illustration.

\begin{figure}[h!]
\centering
\includegraphics[]{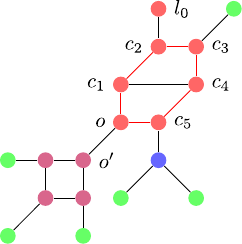}
\caption{Case (1a), $B$ is the subgraph on purple nodes (a \(4\)-cycle). $C$ is the degree-$2$ subgraph on nodes $\{o,c_1,\cdots,c_5\}$ with edges colored red and $L'\setminus\{o'\}=\{l_0\}$.}
\label{fig:1a}
\end{figure}

\item \textbf{(1b)} If $B$ is Type II, then there is a vertex in the induced subgraph on $B$ that has degree $d$. We can then use Proposition \ref{prop: smoothprocedure} to obtain a degree-$2$ subgraph $C'$ of $B$ that contains $o'$ and all vertices that do not have attaching leaves. By deleting $C'$ along with all leaves attached to $C'$, each component of the remaining graph satisfies the assumptions of Lemma \ref{lemma:randomleafavoideigenvalue}. Upon application, we find a further collection of leaves that upon deletion leave the remaining graph with generalized matching polynomial for which $\lambda$ is not a root. See Figure \ref{fig:1b} for an illustration.

\begin{figure}[h!]
\centering
\includegraphics[]{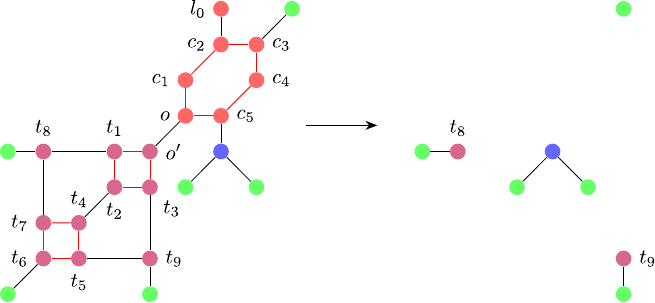}
\caption{Case (1b), $B$ is the subgraph on purple nodes. $C$ is the degree-$2$ subgraph on nodes $\{o,c_1,\cdots,c_5\}$ with edges colored red. $L'\setminus\{o'\}=l_0$. The degree-$2$ subgraph in $B$ that we choose is the graph on nodes $\{o',t_1,\cdots,t_3, t_4, \cdots, t_7\}$ with edges colored red.}
\label{fig:1b}
\end{figure}
\end{itemize}

\textbf{Case 2:} $o$ is not a vertex of $C$. Again after deleting $L'\setminus \{o'\}$ and $C$, all components except for the one containing $B$ and its leaves do not have $\lambda$ as the root of their matching polynomial. The difference this time is that this component contains $o$ (and possibly some other vertices from $G_0$). Now $\{o,o'\}$ is a still a bridge. Define $G_0'$ as the connected component that contains $o$ after deleting the bridge $\{o,o'\}$. There are two sub-cases:
\begin{itemize}
    \item \textbf{(2a)} If $B$ is Type I then we apply Lemma \ref{lemma:lemmaforcontradiction} with $G_0'$ as $G_{\mathrm{co}}$ and $B$ with its attaching leaves as $G_{\mathrm{at}}$ rooted at the vertex $o'$. If $\lambda$ is a root of the generalized matching polynomial of the joined graph $G_{\mathrm{jo}}$ after deleting any collection of the attaching leaves, we conclude that the generalized matching polynomial of both $G_0'$ and $G_0'$ with the bridge from $o'$ to $o$ has $\lambda$ as its root. But this cannot be the case since one of these two is a connected component of $G_0\sqcup\{o'\}\setminus(L'\sqcup C)$ which cannot have $\lambda$ as a root of its generalized matching polynomial by construction. Hence, \(\lambda\) is not a root of the matching polynomial of \(G_{\mathrm{jo}}\), which is the connected component containing \(B\), up to the deletion of some leaves. See Figure \ref{fig:2a} for an illustration.

\begin{figure}[h!]
\centering
\includegraphics[]{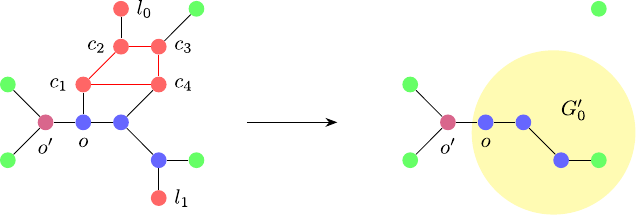}
\caption{Case (2a), $B$ is the subgraph on purple nodes (a singleton). $C$ is the degree-$2$ subgraph on nodes $\{c_1,\cdots,c_4\}$ with edges colored red and $L'\setminus\{o'\}=\{l_0,l_1\}$.}
\label{fig:2a}
\end{figure}
    
    \item \textbf{(2b)} If $B$ is Type II then we again we have two cases depending on whether $o'\in L'$.
    \begin{itemize}
        \item \textbf{(2b1)} If $o'\in L'$, then $\lambda$ is not a root of the generalized matching polynomial of $G_0'$. We thus choose a degree-$2$ subgraph $C'$ of $B$ that contains $o'$ and all vertices that are not attached to leaves by means of Proposition \ref{prop: smoothprocedure}. Then after deleting $C, L'\setminus\{o'\}$ and $C'$ along with all leaves attached to $C'$ from $G$, what remains is a collection of components which either do not have $\lambda$ as a root of their generalized matching polynomial, or they satisfy the assumptions of Lemma~\ref{lemma:randomleafavoideigenvalue}. Hence, after possibly deleting some extra leaves as given by Lemma~\ref{lemma:randomleafavoideigenvalue}, the remaining graph does not have $\lambda$ as a root of its generalized matching polynomial, thus giving a contradiction. See Figure~\ref{fig:2b1} for an illustration.
\begin{figure}[h!]
    \centering
\includegraphics[]{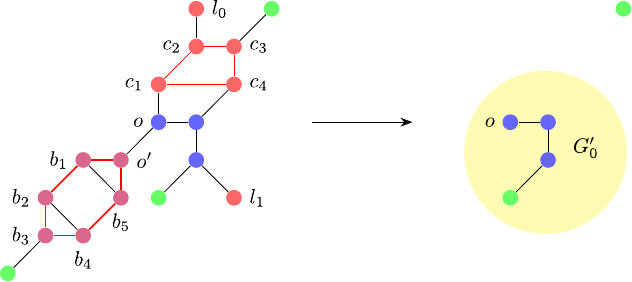}
\caption{Case (2b1), $B$ is the subgraph on purple nodes. $C$ is the degree-$2$ subgraph on nodes $\{c_1,\cdots,c_4\}$ with edges colored red. $L'=\{o',l_0,l_1\}$. The degree-$2$ subgraph in $B$ that we choose is the graph on nodes $\{o',b_1,\cdots,b_5\}$ with edges colored red.}
\label{fig:2b1}
\end{figure}

        \item \textbf{(2b2)} If $o'\notin L'$, then the generalized matching polynomial of $G_0'$ with the bridge $\{o,o'\}$ attached does not have $\lambda$ as a root. We again use Proposition~\ref{prop: smoothprocedure} to find and delete a degree-$2$ subgraph $C'$ from $B$ that does not contain $o'$ but does cover all vertices that do not have leaves attached. After also deleting all leaves attached to $C'$, each connected component of the remaining graph either satisfies the condition of Lemma~\ref{lemma:randomleafavoideigenvalue} or contains $o'$. In the first case, we delete some leaves according to Lemma~\ref{lemma:randomleafavoideigenvalue} to make sure the components don't have $\lambda$ as a root of their generalized matching polynomial. In the latter case, if this component is just $G_0'$ with $o'$ attached, then the contradiction is shown. If not, we treat this component as a new joint graph $G_\mathrm{jo}$ with $G_\mathrm{co}=G_0'$ and a corresponding attached graph with root $o'$. Using Lemma~\ref{lemma:lemmaforcontradiction}, we can delete some leaves to get the desired contradiction. See Figure \ref{fig:2b2} for an illustration.

\begin{figure}[H]
\centering
\includegraphics[]{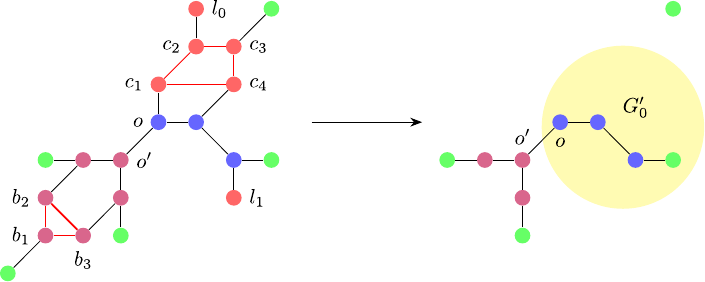}
\caption{Case (2b2), $B$ is the subgraph on purple nodes. $C$ is the degree-$2$ subgraph on nodes $\{c_1,\cdots,c_4\}$ with edges colored red. $L'=\{l_0,l_1\}$. The degree-$2$ subgraph in $B$ that we choose is the graph on nodes $\{b_1,b_2,b_3\}$ with edges colored red.}
\label{fig:2b2}
\end{figure}
    \end{itemize}   
\end{itemize}
\end{proof}

\begin{proof}[Proof of Theorem \ref{thm:regular}]

    Given any odd degree regular graph $G$ that is not bridge-less, by replacing each leaf-block in $G$ by a leaf, the new graph, denoted as $G_\mathrm{int}$ satisfies the conditions of Lemma~\ref{lemma: mainlemma}. Given any $\lambda\in\mathbb{R}$, we may thus find a subset $L_0$ of the leaves of $G_\mathrm{int}$ and a degree-$2$ subgraph $C$ in $G_\mathrm{int}$ such that the generalized matching polynomial of $G_\mathrm{int}\setminus (L_0\sqcup C)$ does not have $\lambda$ as a root. 
    
    The degree-$2$ subgraph $C$ also lives inside $G$. Every leaf in $G_\mathrm{int}$ corresponds to a leaf-block in the bridge-block tree of $G$. For such a leaf-block, we apply Proposition~\ref{prop: smoothprocedure} with $I$ containing only the vertex $v$ incident to the outgoing bridge. This results in two degree-$2$ subgraphs of this block, one containing all vertices and the other containing all vertices but $v$. If $\ell\in L_0$, we delete the former of these subgraphs in the leaf-block from $G$ and if $\ell\notin L_0$, we delete the latter of these subgraphs. The remaining graph coincides with $G_\mathrm{int}\setminus L_0$ up to possibly having some self-loops attached to undeleted leaves. The generalized matching polynomial of this remaining graph coincides with that of $G_\mathrm{int}\setminus L_0$ since  self-loops play no role in matching polynomials. Further deleting $C$ results in a graph whose generalized matching polynomial coincides with $G_\mathrm{int}\setminus (L_0\sqcup C)$. Thus we have obtained a degree-$2$ subgraph $\gamma$ in $G$ whose deletion results in a graph with generalized matching polynomial that does not have $\lambda$ as a root. It follows from Theorem \ref{thm:criteria} that $\lambda$ is not an eigenvalue of $\mathcal{H}^\mathrm{ab}$.
\end{proof}

\appendix

\section{Comparison to the criterion of Banks, Garza-Vargas, and Mukherjee}\label{ap}

Let $G=(V,E)$ be a finite multi-graph and $H=(V_H,E_H)$ a subgraph of $G$. Let $\mathrm{cc}(H)$ be the number of connected components of $H$ and let $\partial H=\partial_GH$ denote the vertex 1-boundary of $H$, that is,
\[
\partial H = \partial_G H := \{ v \in V \setminus V_H: \exists\, u \in V_H \ \text{that is adjacent to $v$} \}.
\]
 If $S$ is a subset of $G$, let $G[S]$ denote the induced subgraph of $G$ on $S$.

Banks, Garza-Vargas and Mukherjee \cite{banks2022point} use the following criterion for $G=(V,E)$ a finite multi-graph and $\lambda \in \mathbb{R}$.\\
{\bf B-GV-M($G,\lambda$)}: There exists $S \subset V$ such that
    \begin{itemize}
        \item $G[S]$ is acyclic, i.e., a forest.
        \item $\lambda$ is an eigenvalue of $\mathcal{H}$ restricted to each connected component of $G[S]$;
        \item The number of boundary points of $G[S]$ is strictly less than the number of its connected components: 
            $$\left|\partial G[S] \right| <\mathrm{cc}(G[S]).$$
    \end{itemize}
\begin{theorem}[{\cite[Corollary 3.4]{banks2022point}}]
\label{thm:criteriaforuniversalcover}
    Let $G$ be a finite multi-graph and $\mathcal{H}$ a Schr\"{o}dinger operator on $G$. Let $\mathcal{H}^\uni$ denote the pullback of $\mathcal{H}$ to the universal cover $\tilde{G}$. 
    Then $\mathcal{H}^\uni$ has $\lambda$ as an eigenvalue if and only if {\bf B-GV-M($G,\lambda$)} holds. 
\end{theorem}

Earlier work was done in this direction by Aomoto \cite{Ao1991}. The eigenvalues in the more general framework of \emph{unimodular trees} have also been studied by Salez \cite{Salez20}.

\begin{proposition}\label{prop:compare}
{\bf B-GV-M($G,\lambda$)} implies the criterion from Theorem~\ref{thm:criteria}: for any degree-2 subgraph $\gamma$ of $G$, $\lambda$ is a root of the generalized matching polynomial $m_{G\setminus\gamma}^\mathcal{H}$.
\end{proposition}

In particular, {\bf B-GV-M($G,\lambda$)} implies \(\lambda\) is also an eigenvalue of \(\mathcal{H}^{\mathrm{ab}}\) by Theorem~\ref{thm:criteria}.

We begin with the following weaker form of Proposition \ref{prop:compare}.
\begin{lemma}\label{lemma:atomcomparision}
If {\bf B-GV-M($G,\lambda$)} holds then $\lambda$ is a root of the generalized matching polynomial $m_G^\mathcal{H}$ of $G$.
\end{lemma}

\begin{proof}
    We proceed by induction on $\left|\partial G[S] \right|$. First, if $\left|\partial G[S] \right|=0$, then $G$ is a disjoint union of the forest $G[S]$ and $G[V\setminus S]$, thus the result holds immediately.

    Now assume that for some $k\geq 1$ the result holds for any $G$ whenever the associated set $S$ has
    $\left|\partial G[S] \right| <k$. Suppose that $\left|\partial G[S] \right| =k$ and fix any $v\in \partial G[S]$. By the recursion Lemma~\ref{lem:recursion}, we have 
    $$m_G^{\mathcal{H}}(x)=\left(x-\mathcal{V}_v\right)\cdot m_{G\setminus \{v\}}^{\mathcal{H}}(x)-\sum_{u\sim v, u\neq v} \left(\sum_{e'\in \vec{E};\, o(e)=v,\,t(e)=u}|w_{e}|^2\right)\cdot m_{G\setminus \{v,u\}}^{\mathcal{H}}(x).$$
   The proof thus follows if we show that $\lambda$ is a root of both $m_{G\setminus \{v\}}^{\mathcal{H}}$ and $m_{G\setminus \{v,u\}}^{\mathcal{H}}$ for $u\sim v, u\neq v$.
   \begin{itemize}
       \item {Proof of $m_{G\setminus \{v\}}^{\mathcal{H}}(\lambda)=0$.} Notice first that the induced subgraph on $S$ in $G\setminus \{v\}$ is the same as the induced subgraph on $S$ in $G$. Moreover, the number of boundary points will be reduced by $1$ after we delete $v$. Thus directly from the induction assumption we have $m_{G\setminus \{v\}}^{\mathcal{H}}(\lambda)=0$.
       
       \item {Proof of $m_{G\setminus \{v,u\}}^{\mathcal{H}}(\lambda)=0$ for $u\sim v, u\neq v$.} If $u\notin S$, then the proof is the same as the case above so we assume that $u \in S$. Define $S_0\subset S$ to be the sub-collection of vertices obtained after removing all vertices from $S$ that are in the same component as $u$ in $G[S]$. If $S_0=\emptyset$ then it must have been the case that $\mathrm{cc}(G[S])=1$ and so $|\partial G[S]|=0$ which is not the case as $k\geq 1$. Thus the induced subgraph on $S_0$ in $G\setminus\{u,v\}$, denoted as $X$, is acyclic and each connected component has eigenvalue $\lambda$. Moreover, the number of its connected component is $\mathrm{cc}(G[S])-1$. Since $$\partial_{G\setminus\{u,v\}}X \sqcup \{v\} \subset \partial_G G[S], $$ we have
       $$\left|\partial_{G\setminus\{u,v\}}X\right|\leq \left|\partial_G G[S]\right|-1< \mathrm{cc}(G[S])-1= \mathrm{cc}(X).$$
       By the induction assumption it follows that $m_{G\setminus \{v,u\}}^{\mathcal{H}}(\lambda)=0$.
   \end{itemize}
\end{proof}

%To prove Proposition  \ref{prop:compare} we are required to improve Lemma \ref{lemma:atomcomparision} to show that $\lambda$ is not just a root of $m_G^\mathcal{H}(x)$ but also of $m_{G\setminus\gamma}^\mathcal{H}(x)$ for any degree-2 subgraph $\gamma$ of $G$. This will however follow from application of Lemma \ref{lemma:atomcomparision} to an appropriately constructed subset of $V$.

\begin{proof}[Proof of Proposition \ref{prop:compare}]
Suppose  {\bf B-GV-M($G,\lambda$)} holds and let $S$ be the set provided.
Let $\gamma_0$ be a degree-2 subgraph of $G$. By Lemma \ref{lemma:atomcomparision} it suffices to prove \\{\bf B-GV-M($G \backslash \gamma_0, \lambda$)}.

To this end, define $S_0\subseteq S$ to be the sub-collection of vertices that are in connected components of $G[S]$ disjoint from $\gamma_0$. 

We claim that if $\gamma_0$ transverses exactly $k$ disjoint components of $G[S]$, then 
$$|\partial_G G[S] \cap V_{\gamma_0}| \geq k.$$
Note that this shows in particular that $S_0\neq\emptyset$ since otherwise we would have 
$$\mathrm{cc}(G[S])>\partial_G G[S]\geq|\partial_G G[S] \cap V_{\gamma_0}| \geq k=\mathrm{cc}(G[S]),$$
a contradiction.

\begin{figure}[H]
\centering
\includegraphics[]{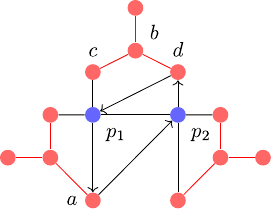}
\caption{Red nodes: the Aomoto set \(S\). Blue nodes: \(\partial_G G[S]\). The cycle on \(\{a,p_1,p_2\}\) transverses \(1\) component of the Aomoto set and \(2\) nodes from \(\partial_G G[S]\). The cycle on \(\{a,p_2,d,p_1\}\) transverses \(2\) components of the Aomoto set and \(2\) nodes from \(\partial_G G[S]\). }
\label{fig:appendix}
\end{figure}

{Proof of the claim}:
Each connected component of $\gamma_0$ is a cycle, so we can arbitrarily fix an orientation on each of these cycles, thereby assigning a direction to their edges. For each component \(T\) of \(G[S]\) traversed by \(\gamma_0\), choose a vertex \(u_T \in \partial_G G[S] \cap V_{\gamma_0}\) such that there exists a directed edge in \(\gamma_0\) from some \(v \in V_T\) to \(u_T\) that is compatible with the chosen orientation. Such a vertex \(u_T\) exists because \(T\) is a tree, and hence cannot contain an entire cycle. Note that this \(u_T\) might not be unique. If \(T_1\) and \(T_2\) are two distinct components of \(G[S]\), then \(u_{T_1} \neq u_{T_2}\), since the orientation ensures that each vertex in \(\partial_G G[S]\) receives at most one incoming edge from \(G[S]\). Therefore, this choice of \(u_T\) defines an injection from the set of connected components of \(G[S]\) traversed by \(\gamma_0\) to \(\partial_G G[S]\cap V_{\gamma_0}\). This proves the claim. See Figure~\ref{fig:appendix} for an illustration.

%Each connected component of $\gamma_0$ is a cycle and so we can arbitrarily fix an orientation on each of these cycles, giving a direction to their edges. We define a bipartite graph $G'=(V_1\sqcup V_2, E')$ as follows: Let $V_1$ be the collection of the $k$ disjoint components of $G[S]$ that $\gamma_0$ transverses, and let $V_2=\partial_G G[S] \cap V_{\gamma_0}$. Put an edge between $x\in V_1$ and $u\in V_2$ if there exists a vertex $v_x$ in the component of $G[S]$ that $x$ represents, such that there is a directed edge originating from $v_x$ to $u$ that is compatible with the given orientation. Since each connected component of $G[S]$ is a tree, all vertices from $V_1$ are of degree equal or larger than $1$ since a cycle in $\gamma_0$ can not live completely within a tree. However, since $\gamma_0$ is a collection of disjoint cycles, all vertices from $V_2$ are of degree at most $1$. Thus we have $|V_2| \geq |V_1|=k$ and the claim follows.

Now let $X$ be the induced subgraph on $S_0$ in $G\setminus\gamma_0$. Since 
$$\partial_{G\setminus\gamma_0}X \sqcup (\partial_G G[S] \cap V_{\gamma_0}) \subset \partial_G G[S],$$
we have 
$$\left|\partial_{G\setminus\gamma_0}X\right|\leq \left|\partial_G G[S]\right|-|\partial_G G[S] \cap V_{\gamma_0}|\leq \left|\partial_G G[S]\right|-k<\mathrm{cc}(G[S])-k= \mathrm{cc}(X).$$ 
Moreover, since $S_0$ is non-empty and each of the connected components of $X$ is a connected component of $G[S]$, we have by hypothesis that $\lambda$ is an eigenvalue of each of the components of $X$. This proves {\bf B-GV-M($G \backslash \gamma_0 ,\lambda$)} as required.
\end{proof}

\bibliographystyle{abbrv}
\bibliography{reference}

\begin{tabular}{@{}l@{}}%
Wenbo Li \\
School of Mathematical Sciences,
University of Science and Technology of China, \\
No.\,96 Jinzhai Road, Hefei, China \\
\texttt{patlee@mail.ustc.edu.cn}
\end{tabular}

\vspace{1em}

\begin{tabular}{@{}l@{}}%
Michael Magee \\
Department of Mathematical Sciences, 
Durham University, DH1 3LE Durham, UK \\
\texttt{michael.r.magee@durham.ac.uk}
\end{tabular}

\vspace{1em}

\begin{tabular}{@{}l@{}}%
Mostafa Sabri \\
Science Division,
New York University Abu Dhabi, Saadiyat Island, UAE \\
\texttt{mostafa.sabri@nyu.edu}
\end{tabular}

\vspace{1em}

\begin{tabular}{@{}l@{}}%
Joe Thomas \\
Department of Mathematical Sciences,
Durham University, DH1 3LE Durham, UK \\
\texttt{joe.thomas@durham.ac.uk}
\end{tabular}

\end{document}